\documentclass[reqno]{amsart}
\usepackage{amssymb, amsmath, amsthm, amscd, mathrsfs}

\usepackage{paralist}
\usepackage{cite}
\usepackage{bigints}
\usepackage{dsfont}
\usepackage{empheq}
\usepackage{amssymb}
\usepackage{cases}
\usepackage{enumitem}
\allowdisplaybreaks

\usepackage{caption}
\usepackage{booktabs}

\usepackage{float}
\usepackage[ruled,vlined,linesnumbered,noresetcount]{algorithm2e}

\usepackage{hyperref}

\theoremstyle{plain}
\newtheorem{theorem}{Theorem}[section]

\newtheorem{lemma}[theorem]{Lemma}

\theoremstyle{definition}
\newtheorem{definition}[theorem]{Definition}
\newtheorem{remark}{Remark}

\def \dv {\mathrm{div}}
\def \d {\mathrm{d}}

\title[Impulsive heat equation with dynamic boundary conditions] 
      {Impulse null approximate controllability for heat equation with dynamic boundary conditions}

\author{S. E. Chorfi}
\author{G. El Guermai}
\author{L. Maniar}
\author{W. Zouhair}

\address{S. E. Chorfi, G. El Guermai, L. Maniar and W. Zouhair, Cadi Ayyad University, Faculty of Sciences Semlalia, LMDP, UMMISCO (IRD-UPMC), B.P. 2390, Marrakesh, Morocco}

\email{chorphi@gmail.com, ghita.el.guermai@gmail.com, maniar@uca.ma, walid.zouhair.fssm@gmail.com}

\makeatletter 
\@namedef{subjclassname@2020}{%
  \textup{2020} Mathematics Subject Classification}
\makeatother


\begin{document}
\begin{abstract}
The main purpose of this article is to prove a logarithmic convexity estimate for the solution of a linear heat equation subject to dynamic boundary conditions in a bounded convex domain. As an application, we prove the impulsive null approximate controllability for an impulsive heat equation with dynamic boundary conditions.
\end{abstract}

\maketitle

\section{Introduction}
It is well-known that the heat equation is one of the most significant partial differential equations of parabolic type. It is a model for a large class of physical phenomena which describes the distribution of heat of a homogeneous medium over time.

Let $\Omega\subset \mathbb{R}^n$ be a bounded convex domain of smooth boundary $\Gamma$ of class $C^2$ and let $T>0$ be a final fixed time. More precisely, we consider the following controlled system
\begin{empheq}[left = \empheqlbrace]{alignat=2} \label{1.1}
\begin{aligned}
&\partial_{t} \psi-\Delta \psi=0, && \qquad\text { in } \Omega \times(0, T) \backslash\{\tau\},\\
&\psi(\cdot, \tau)=\psi\left(\cdot, \tau^{-}\right)+\mathds{1}_{\omega} h(\cdot,\tau), && \qquad\text { in } \Omega,\\
&\partial_{t}\psi_{\Gamma} - \Delta_{\Gamma} \psi_{\Gamma} + \partial_{\nu}\psi =0, && \qquad\text { on } \Gamma \times(0, T)\backslash\{\tau\}, \\
&\psi_{\Gamma}(\cdot, \tau)=\psi_{\Gamma}\left(\cdot, \tau^{-}\right), && \qquad\text { on } \Gamma,\\
& \psi_{\Gamma}(x,t) = \psi_{|\Gamma}(x,t), &&\qquad\text{ on } \Gamma \times(0, T) , \\
& \left(\psi(\cdot, 0),\psi_{\Gamma}(\cdot, 0)\right)=\left(\psi^{0},\psi^{0}_{\Gamma}\right), && \qquad \text{ on } \Omega\times\Gamma,
\end{aligned}
\end{empheq}
where $\left(\psi^{0},\psi^{0}_{\Gamma}\right)\in \mathbb{L}^2$ denotes the initial condition, $\psi(\cdot,\tau^{-})$ denotes the left limit of the function $\psi$ at time $\tau \in (0,T)$, $\omega\Subset \Omega$ is a nonempty open subset, and $\mathds{1}_\omega$ stands for the characteristic function of $\omega$. We also denote by $\psi_{|\Gamma}$ the trace of $\psi$, and $\partial_{\nu} \psi:=(\nabla \psi \cdot \nu)_{|\Gamma}$ denotes the normal derivative, where $\nu$ stands for the unit exterior field normal to $\Gamma$. The tangential gradient will be denoted as $\nabla_{\Gamma} \psi=\nabla \psi-\left(\partial_{\nu} \psi\right) \nu$. Let $\mathrm{g}$ denote the standard Riemannian metric on $\Gamma$ inherited from $\mathbb{R}^n$. The Laplace-Beltrami operator $\Delta_\Gamma$ is defined locally as follows
\begin{equation*}
\Delta_\Gamma=\frac{1}{\sqrt{|\mathrm{g}|}} \sum_{i,j=1}^{N-1} \frac{\partial}{\partial x^i} \left(\sqrt{|\mathrm{g}|}\, \mathrm{g}^{ij} \frac{\partial}{\partial x^j}\right),\label{eqlb}
\end{equation*}
where $\mathrm{g}=\left(\mathrm{g}_{ij}\right)$ is the metric tensor, $\mathrm{g}^{-1}=\left(\mathrm{g}^{ij}\right)$ its inverse and $|\mathrm{g}|=\det\left(\mathrm{g}_{ij}\right)$. Finally, we will denote the Hessian matrix of $\psi$ with respect to $\mathrm{g}$ by $\nabla^2_\Gamma \psi$. In the sequel, we mainly use the divergence formula
\begin{equation*}
\int_\Gamma \Delta_\Gamma u\, v \,\d S =- \int_\Gamma \langle \nabla_\Gamma u, \nabla_\Gamma v\rangle_\Gamma \,\d S, \quad u\in H^2(\Gamma), \, v\in H^1(\Gamma), \label{sdt}
\end{equation*}
where $\langle \cdot, \cdot \rangle_\Gamma$ is the Riemannian inner product of tangential vectors on $\Gamma$. We refer to \cite{Jo'08, Ta'11} for a detailed exposition of PDEs on Riemannian manifolds.

The observability estimate for parabolic equations was initiated independently by \cite{LGLR} and \cite{FI'96} in the context of null controllability, based on Carleman inequalities (see also \cite{FG'06}). Furthermore, the observability estimate at one point of time, which is an estimate of the energy at a fixed time on the whole domain in terms of the energy at the same time but on a small subdomain, was already established for a linear heat equation with homogeneous Dirichlet boundary conditions by Phung in \cite{pkm}, using a new strategy combining the logarithmic convexity method and a Carleman commutator approach  (see \cite{ACM'21,Pa'75}). From the underlying observation at one time, many applications were derived as impulse approximate controllablity \cite{QSGW,pkdgwyx}, which consists on finding an impulsive control steering approximately the system from an arbitrary initial state to a final state in a finite interval of time. This type of control is  very weak since it only acts in a subdomain at one instant of time, which makes the problem of controllablity for the heat equation with impulse control very challenging. Unlike the interior controllability for impulsive systems that have been extensively studied in the literature, see for instance \cite{CDJUHLOC,ACCDGHL,GL,CDGHL,LHWZME} and the references therein, the problem of controllability with impulse controls has attracted less attention and not as many works are available in this area, we mention \cite{ABWZ,ka,pkdgwyx}.

Parabolic equations with dynamic boundary conditions have received a considerable attention in last years, see, e.g., \cite{FGGR'02, KMMR'19, VV'11}. Such equations arise in several areas of applications which include chemical reaction theory and heat transfer problems as well as populations dynamic \cite{FH'11,La'32}. We mention the papers \cite{Go'06, Sa'20}, which physically derive the dynamic boundary conditions. Furthermore, the controllability and inverse problems of heat equation with dynamic boundary conditions have recently been investigated in \cite{ACM1'21', ACMO'20,BCMO'20,MMS'17,KM'19}, where the authors have proven controllability and stability results by proving new Carleman estimates.

To the best of authors's knowledge, there is no paper which deals with heat equation with dynamic boundary conditions in the presence of an impulsive control. Motivated by the above facts, we study the approximate controllability for the impulsive heat equation with dynamic boundary conditions.

The key result that will enable us to prove the impulsive approximate controllability of the above system is the following observability estimate at one point of time.
\begin{theorem}\label{thm1.1}
Let $\Omega \subset \mathbb{R}^n$  be a bounded convex domain of class $C^2$, $\omega \Subset \Omega$. Let $\langle \cdot, \cdot\rangle$ denote the usual inner product in
$L^2(\Omega) \times L^2 (\Gamma)$ and let $\|\cdot \|$ be its corresponding norm. Then the following observation estimate holds
\begin{equation}\label{1.2}
\|U(\cdot, T)\| \leq\left(\mu \mathrm{e}^{\frac{K}{T}}\|u(\cdot, T)\|_{L^{2}(\omega)}\right)^{\beta}\|U(\cdot, 0)\|^{1-\beta},
\end{equation}
where $\mu, K >0$, $\beta \in (0,1)$ and $U=\left(u,u_{\Gamma}\right)$ is the solution of the following system
\begin{empheq}[left = \empheqlbrace]{alignat=2}\label{1.3}
\begin{aligned}
&\partial_{t} u-\Delta u=0, && \qquad\text { in } \Omega \times(0, T), \\
&\partial_{t}u_{\Gamma} - \Delta_{\Gamma} u_{\Gamma} + \partial_{\nu}u =0, && \qquad\text { on } \Gamma \times(0, T), \\
& u_{\Gamma}(x,t) = u_{|\Gamma}(x,t), &&\qquad\text{ on } \Gamma \times(0, T) , \\
& \left(u(\cdot, 0),u_{\Gamma}(\cdot, 0)\right)=\left(u^{0},u^{0}_{\Gamma}\right), && \qquad \text{ on } \Omega\times\Gamma.
\end{aligned}
\end{empheq}
\end{theorem}

\begin{remark}
We emphasize that \eqref{1.2} is an observability inequality estimating the whole solution $U=(u,u_\Gamma)$ of the system \eqref{1.3} at final time $T$ by only using one internal observation on the first component $u$, which is localized in the subdomain $\omega$.
\end{remark}

Consequently, we obtain the main result which reads as follows
\begin{theorem}\label{thm1.2}
The system \eqref{1.1} is null approximate impulse controllable at any time $T > 0$. Moreover, for any $\varepsilon > 0,$ the cost of null approximate impulse control function at time $T$ satisfies $K(T, \varepsilon) \leq \frac{\mathcal{M}_{1} \mathrm{e}^{\frac{\mathcal{M}_{2}}{T-\tau}}}{\varepsilon^{\delta}},$ where the positive constants $\mathcal{M}_1$, $\mathcal{M}_2$ and $\delta$ are from the estimate \eqref{2.38.}.
\end{theorem}

This paper is organized as follows. In Section 2, we recall some preliminary results needed in the sequel. In Section 3, we first introduce the weight function; then we present the strategy to obtain the observation estimate at one point of time by a Carleman commutator approach. Section 4 is devoted to the proof of the impulsive approximate controllability for the impulsive heat equation with dynamic boundary conditions.

\section{Preliminaries}
In this section, we recall some results that will be useful in the sequel. We will often use the following real space
$$\mathbb{L}^2:=L^2(\Omega, \d x)\times L^2(\Gamma, \d S),$$
which is a Hilbert space equipped with the inner product given by
\begin{align*}
\langle (u,u_\Gamma),(v,v_\Gamma)\rangle =\langle u,v\rangle_{L^2(\Omega)} +\langle u_\Gamma,v_\Gamma\rangle_{L^2(\Gamma)},
\end{align*}
where the Lebesgue measure on $\Omega$ is denoted by $\d x$ and the surface measure on $\Gamma$ by $\d S$. We also consider the space
$$\mathbb{H}^k:=\left\{(u,u_\Gamma)\in H^k(\Omega)\times H^k(\Gamma)\colon u_{|\Gamma} =u_\Gamma \right\} \text{ for } k=1,2,$$ equipped with the standard product norm.

System \eqref{1.3} can be written as an abstract Cauchy problem
\begin{equation}\label{acp}
\begin{cases}
\hspace{-0.1cm} \partial_t \mathbf{U}=\mathbf{A} \mathbf{U}, \quad 0<t \le T, \nonumber\\
\hspace{-0.1cm} \mathbf{U}(0)=(u^0, u^0_\Gamma), \nonumber
\end{cases}
\end{equation}
where $\mathbf{U}:=(u,u_{\Gamma})$ and the linear operator $\mathbf{A} \colon D(\mathbf{A}) \subset \mathbb{L}^2 \longrightarrow \mathbb{L}^2$ is given by
\begin{equation*}
\mathbf{A}=\begin{pmatrix} \Delta & 0\\ -\partial_\nu & \Delta_\Gamma \end{pmatrix}, \qquad \qquad D(\mathbf{A})=\mathbb{H}^2. \label{E21}
\end{equation*}
The operator $\mathbf{A}$ is selfadjoint and dissipative, it generates then an analytic $C_0$-semigroup of contractions $\left(\mathrm{e}^{t\mathbf{A}}\right)_{t\geq 0}$ on $\mathbb{L}^2$ (see \cite{MMS'17}). Hence, the solution map $t\mapsto \mathrm{e}^{t\mathbf{A}} \mathbf{U}_0$ is infinitely many times differentiable for $t>0$ and $\mathrm{e}^{t\mathbf{A}} \mathbf{U}_0 \in D\left(\mathbf{A}^m\right)$ for every $\mathbf{U}_0 \in \mathbb{L}^2$ and $m\in \mathbb{N}$.

On the other hand, we rewrite system \eqref{1.1} as the impulsive Cauchy problem
\begin{equation}\label{ACP}
\text{(ACP)} \;\; \begin{cases}
\hspace{-0.1cm} \partial_t \Psi(t)=\mathbf{A} \Psi(t), \quad (0,T)\setminus\{\tau\}, \nonumber\\
\hspace{-0.1cm} \bigtriangleup \Psi(\tau) = (\mathds{1}_{\omega} h(\tau), 0) ,\nonumber\\
\hspace{-0.1cm} \Psi(0)=(\psi^0, \psi^0_\Gamma), \nonumber
\end{cases}
\end{equation}
where $\Psi:=(\psi,\psi_{\Gamma})$ and $\bigtriangleup \Psi(\tau) :=\Psi\left(\cdot,\tau\right) - \Psi\left(\cdot,\tau^{-}\right)$.
For all $\Psi_{0}:=(\psi^0, \psi^0_\Gamma) \in \mathbb{L}^2$, the system (ACP) has a unique mild solution (see \cite{VLDDPSS, VSRKGJSPM}) given by
$$
\Psi(t) = \mathrm{e}^{t\mathbf{A}} \Psi_{0} + \mathds{1}_{\{t\geq \tau \}}(t)\, \mathrm{e}^{(t-\tau)\mathbf{A}} (\mathds{1}_{\omega} h(\tau),0), \qquad t\in (0,T).
$$

To achieve the aim of this paper, we  need the following lemma.
\begin{lemma}\label{lem2.1}
Let $\Omega \subset \mathbb{R}^n$ be a bounded convex domain of class $C^1$ and $x_0 \in \Omega$.
\begin{itemize}
\item[(i)] There exists a convex function $\theta:\mathbb R^n\rightarrow \mathbb R$ of class $C^1$ such that
\begin{align*}
\Omega &=\left\{x\in \mathbb R^n, \theta(x)<0\right\},\\
\partial \Omega &=\left\{x\in \mathbb R^n, \theta(x)=0\right\},\\
\nabla&\theta(x) \not=0 \quad \text{ for all } x\in \Gamma.
\end{align*}
\item[(ii)] For all $x\in \Gamma$,
\begin{equation*}
-(x-x_0)\cdot \nu(x)<0. \label{nd}
\end{equation*}
\end{itemize}
\end{lemma}

\begin{proof}
$(i)$ For simplicity, we assume that $x_0=0$. The gauge function of $\Omega$,  defined by
$$
j_\Omega(x) =\inf \left\{t>0: \frac{x}{t} \in \Omega\right\},
$$
is convex positively homogeneous of degree $1$. We recall here some well-known properties of $j_\Omega$, see  \cite{Ho'94} 
\begin{itemize}
\item[(a)] $j_{\Omega}(0)=0,\;  j_{\Omega}(x)>0$, $\;\forall x \neq 0$,
\item[(b)] $j_{\Omega}(\lambda x)=\lambda j_{\Omega}(x)$, $\;\forall \lambda>0$, $\;\forall x \in \mathbb{R}^{N}$,
\item[(c)] $j_{\Omega}(x+y)\le j_{\Omega}(x)+ j_{\Omega}(y)$, $\;\forall x,y \in \mathbb{R}^n$,
\item[(d)] $j_\Omega$ is Lipschitz continuous,
\item[(e)] $\Omega=\left\{x \in \mathbb{R}^{N} : j_{\Omega}(x)<1\right\}$,
\item[(f)] $\Gamma = \left\{x \in \mathbb{R}^{N} : j_{\Omega}(x)=1\right\}$.
\end{itemize}
The boundary $\Gamma$ is parameterized by the function
\begin{align*}
 g: \mathbb S^{N-1} &\rightarrow \Gamma\\
  x &\mapsto \frac{x}{j_\Omega(x)}.
\end{align*}
Since $\Gamma$ is of class $C^1$, $j_\Omega$ is a $C^1$ function on $\mathbb{R}^n\setminus \{0\}$. Moreover, since $j_\Omega$ is Lipschitz continuous, $\nabla j_\Omega$ is bounded. Let $\theta = j_\Omega^2-1$. We have $\Omega= \{ \theta < 0\}$ and $\Gamma = \{ \theta = 0\}$. Moreover, $\theta$ is of class $C^1$ on $\mathbb R^n \setminus\{0\}$. On the other hand, we have $\nabla\theta = 2 j_\Omega\nabla j_\Omega$ and $\lim\limits_{x\to 0} \nabla\theta(x)=0$. Finally, if there exists $x\in \Gamma$ such that $\nabla \theta(x)=0$, then $\nabla j_\Omega(x)=0$. By the subgradient inequality, we obtain that
$$j_\Omega(y) \ge j_\Omega(x) + \nabla j_\Omega(x) (y-x) \ge 1$$
for all $y\in \mathbb{R}^n$, which is not the case.

$(ii)$ Let $x\in \Gamma$. Then $\theta(x)=0$, $\theta(x_0)<0$ and $\nu=\frac{\nabla \theta}{|\nabla \theta|}$. Since $\theta$ is convex, by the subgradient inequality, we obtain
\begin{align*}
-(x-x_0)\cdot\nabla \theta (x) \le \theta(x_0) < 0.
\end{align*}
Thus, the proof is complete.
\end{proof}
\bigskip

\section{Logarithmic convexity by the Carleman commutator approach}
In this section, we will present an approach to obtain the observation estimate at one point of time for system \eqref{1.3} in a convex bounded domain $ \Omega \subset \mathbb R^n$ with boundary $\Gamma = \partial \Omega$ of class $C^2$. We start by introducing the weight function needed in the proof. Following \cite{RBKDP}, we invoke a small parameter $s \in (0,1)$ that enables us to deal with some new boundary terms arising from integration by parts.

\subsection{The weight function}
Let $\Phi \colon \overline{\Omega} \times (0,T)\rightarrow \mathbb{R} $ be the weight function defined by
\[
\Phi (x,t)=\frac{-s\left\lvert x-x_{0}\right\lvert^2}{4(T-t+h)},
\]
where $x_0 \in \omega \Subset \Omega, h > 0$ and $s \in \left(0,1\right)$. To simplify, we set
\[
\varphi(x)=\frac{-\left\lvert x-x_{0}\right\lvert^2}{4}\quad \text{and}\quad \Upsilon(t)=T-t+ h,
\]
so that
$$\Phi (x,t)=\frac{s \varphi(x)}{\Upsilon(t)}.$$
The function $\varphi$ satisfies the following properties
\begin{itemize}
\item[(1)] $\varphi(x)+\left\lvert \nabla\varphi(x)\right\lvert^{2} = 0,\; \forall x\in \overline{\Omega}$,
\item[(2)] $\nabla\varphi(x)=-\dfrac{1}{2}(x-x_{0}), \;\forall x\in\overline{\Omega}$,
\item[(3)] $\Delta\varphi(x) =-\dfrac{n}{2},\; \forall x\in\overline{\Omega}$,
\item[(4)] $\nabla^2\varphi=-\dfrac{1}{2}I_{n}$,
\item[(5)] $\partial_{\nu}\varphi(x)=-\dfrac{1}{2}(x-x_{0})\cdot \nu(x),\; \forall x\in \Gamma$.
\end{itemize}
\subsection{Proof of the logarithmic convexity estimate}
We shall present the strategy of the proof step by step following some ideas in \cite{pkm} in a modified form.

\noindent\textbf{Step 1. Symmetric part and antisymmetric part.} Let $\left(u^{0}, u_{\Gamma}^{0}\right) \in \mathbb{L}^{2} \backslash\{(0,0)\}$. We define the quantity 
\begin{equation*}
F(x, t)=U(x, t) \mathrm{e}^{\Phi(x, t) / 2},
\end{equation*}
where $U=\left(\begin{array}{c}u \\ u_{\Gamma}\end{array}\right)$
is the solution of \eqref{1.3} and $ F =\begin{pmatrix}
f \\ f_{\Gamma}
\end{pmatrix}$. We introduce the operator $P$ as follows
\begin{equation*}
P F = \mathrm{e}^{\Phi / 2}
\begin{pmatrix}
\partial_{t}-\Delta & 0 \\[3mm]
\partial_{\nu} & \partial_{t}-\Delta_{\Gamma}
\end{pmatrix} \mathrm{e}^{-\Phi / 2} F.
\end{equation*}
Then
\begin{equation*}
P F =
\begin{pmatrix}
\partial_{t} f-\Delta f-\frac{1}{2}f\left(\partial_{t} \Phi+\frac{1}{2}|\nabla \Phi|^{2}\right)+ \nabla \Phi \cdot\nabla f +\frac{1}{2}\Delta \Phi f \\[3mm]
\partial_{t} f_{\Gamma}-\Delta_{\Gamma} f_{\Gamma}+ \partial_{\nu} f -\frac{1}{2} f_{\Gamma}\left(\partial_{t} \Phi+\frac{1}{2}|\nabla_{\Gamma} \Phi|^{2}\right)-\frac{1}{2} \partial_{\nu}\Phi f_{\Gamma} +\frac{1}{2}\Delta_{\Gamma} \Phi f_{\Gamma} + \langle \nabla_{\Gamma} \Phi ,\nabla_{\Gamma} f_{\Gamma}\rangle_{\Gamma}
\end{pmatrix}.
\end{equation*}
Let us define $P_{1}$ as follows
\begin{equation*}
P_{1} F =
\begin{pmatrix}
\Delta f+\frac{1}{2}f\left(\partial_{t} \Phi+\frac{1}{2}|\nabla \Phi|^{2}\right)- \nabla \Phi\cdot \nabla f -\frac{1}{2}\Delta \Phi f \\[3mm]
\Delta_{\Gamma} f_{\Gamma}- \partial_{\nu} f +\frac{1}{2} f_{\Gamma}\left(\partial_{t} \Phi+\frac{1}{2}|\nabla_{\Gamma} \Phi|^{2}\right)+\frac{1}{2} \partial_{\nu}\Phi f_{\Gamma} -\frac{1}{2}\Delta_{\Gamma} \Phi f_{\Gamma} - \langle \nabla_{
\Gamma} \Phi ,\nabla_{\Gamma} f_{\Gamma}\rangle_{\Gamma}
\end{pmatrix}.
\end{equation*}
Since $U$ is the solution of \eqref{1.3}, then $P F =0$. Hence
\begin{equation*}
\begin{pmatrix}
\partial_{t} f \\
\partial_{t} f_{\Gamma}
\end{pmatrix} = P_{1} F.
\end{equation*}
Let us compute the adjoint operator of $P_1$. For any $G = \begin{pmatrix}
g \\
g_{\Gamma}
\end{pmatrix} \in  \mathbb{H}^1$, we have
\begin{align*}
\left\langle P_1 F , G \right\rangle &=\displaystyle\int_{\Omega} \Delta f g \mathrm{d} x+\int_{\Omega} \frac{1}{2} f g\left(\partial_{t} \Phi+\frac{1}{2}|\nabla \Phi|^{2}\right) \mathrm{d}x- \int_{\Omega} g \nabla \Phi \cdot \nabla f \mathrm{d} x\\
&-\displaystyle\frac{1}{2} \int_{\Omega}  \Delta \Phi fg \mathrm{d}x + \int_{\Gamma}  \Delta_{\Gamma} f_{\Gamma}g_{\Gamma} \mathrm{d} S - \int_{\Gamma}  \partial_{\nu} f g_{\Gamma} \mathrm{d} S\\
&+\displaystyle \frac{1}{2} \int_{\Gamma}  f_{\Gamma} g_{\Gamma}\left(\partial_{t} \Phi+\frac{1}{2}|\nabla_{\Gamma} \Phi|^{2}\right)\mathrm{d}S+\frac{1}{2} \int_{\Gamma}  \partial_{\nu}\Phi f_{\Gamma} g_{\Gamma} \mathrm{d} S\\
 &-\displaystyle \frac{1}{2} \int_{\Gamma} \Delta_{\Gamma} \Phi f_{\Gamma} g_{\Gamma} \mathrm{d} S - \int_{\Gamma} g_{\Gamma}\left\langle\nabla_{\Gamma} \Phi, \nabla f_{\Gamma}\right\rangle_{\Gamma} \mathrm{d} S\\
&= \displaystyle\int_{\Gamma} \partial_{\nu} f g_{\Gamma} \mathrm{d} S - \int_{\Omega} \nabla f  \cdot\nabla g \mathrm{d} x  +\frac{1}{2} \int_{\Omega}  f g\left(\partial_{t} \Phi+\frac{1}{2}|\nabla \Phi|^{2}\right) \mathrm{d}x\\
&-\displaystyle\int_{\Gamma}  \partial_{\nu}\Phi f_{\Gamma} g_{\Gamma} \mathrm{d} S+\int_{\Omega} f\, \dv(\nabla \Phi g )\mathrm{d}x - \frac{1}{2} \int_{\Omega} \Delta \Phi f g \mathrm{d}x \\
&-\displaystyle \int_{\Gamma} \left\langle\nabla_{\Gamma}f_{\Gamma} , \nabla_{\Gamma} g_{\Gamma}\right\rangle_{\Gamma} \mathrm{d} S - \int_{\Gamma} \partial_{\nu} f g_{\Gamma} \mathrm{d}S+\frac{1}{2} \int_{\Gamma}  f_{\Gamma} g_{\Gamma}\left(\partial_{t} \Phi+\frac{1}{2}|\nabla_{\Gamma} \Phi|^{2}\right) \mathrm{d}S\\
&+ \displaystyle \frac{1}{2} \int_{\Gamma}  \partial_{\nu}\Phi f_{\Gamma} g_{\Gamma} \mathrm{d} S -\frac{1}{2} \int_{\Gamma} \Delta_{\Gamma} \Phi f_{\Gamma} g_{\Gamma} \mathrm{d} S +\int_{\Gamma} f_{\Gamma} div_{\Gamma} (g_{\Gamma}\nabla_{\Gamma} \Phi) \mathrm{d} S\\
&= - \displaystyle\int_{\Gamma}  f_{\Gamma} \partial_{\nu} g \mathrm{d} S+ \int_{\Omega}  f  \Delta g \mathrm{d} x  +\frac{1}{2} \int_{\Omega}  f g\left(\partial_{t} \Phi+\frac{1}{2}|\nabla \Phi|^{2}\right) \mathrm{d}x\\
&-\displaystyle \frac{1}{2}\int_{\Gamma}  \partial_{\nu}\Phi f_{\Gamma} g_{\Gamma} \mathrm{d} S+\int_{\Omega} \Delta \Phi f g \mathrm{d}x + \int_{\Omega} f \nabla \Phi\cdot \nabla g \mathrm{d}x -\frac{1}{2}\int_{\Omega} \Delta \Phi f g \mathrm{d}x\\
&+\displaystyle   \int_{\Gamma} f_{\Gamma} \Delta_{\Gamma} g_{\Gamma} \mathrm{d} S+\frac{1}{2} \int_{\Gamma}  f_{\Gamma} g_{\Gamma}\left(\partial_{t} \Phi+\frac{1}{2}|\nabla_{\Gamma} \Phi|^{2}\right) \mathrm{d}S\\
&- \displaystyle  \frac{1}{2} \int_{\Gamma}  \Delta_{\Gamma}\Phi f_{\Gamma} g_{\Gamma} \mathrm{d} S+\int_{\Gamma} f_{\Gamma}\left\langle\nabla_{\Gamma} \Phi_{\Gamma} , \nabla_{\Gamma} g_{\Gamma}\right\rangle_{\Gamma} \mathrm{d} S+ \int_{\Gamma} \Delta_{\Gamma} \Phi f_{\Gamma} g_{\Gamma} \mathrm{d} S\\
&= - \displaystyle\int_{\Gamma}  f_{\Gamma} \partial_{\nu} g \mathrm{d} S + \int_{\Omega}  f  \Delta g \mathrm{d} x  +\frac{1}{2} \int_{\Omega}  f g\left(\partial_{t} \Phi+\frac{1}{2}|\nabla \Phi|^{2}\right) \mathrm{d}x \\
&-\displaystyle \frac{1}{2}\int_{\Gamma}  \partial_{\nu}\Phi f_{\Gamma} g_{\Gamma} \mathrm{d} S+\frac{1}{2} \int_{\Omega} \Delta \Phi f g \mathrm{d}x + \int_{\Omega} f \nabla \Phi\cdot \nabla g \mathrm{d}x \\
&+ \displaystyle\int_{\Gamma} f_{\Gamma} \Delta_{\Gamma} g_{\Gamma} \mathrm{d} S
+ \frac{1}{2} \int_{\Gamma}  f_{\Gamma} g_{\Gamma}\left(\partial_{t} \Phi+\frac{1}{2}|\nabla_{\Gamma} \Phi|^{2}\right) \mathrm{d} S \\
&+\displaystyle \frac{1}{2} \int_{\Gamma}  \Delta_{\Gamma}\Phi f_{\Gamma} g_{\Gamma} \mathrm{d} S+\int_{\Gamma} f_{\Gamma}\left\langle\nabla_{\Gamma}\Phi_{\Gamma} , \nabla_{\Gamma} g_{\Gamma}\right\rangle_{\Gamma} \mathrm{d} S\\
&= \left\langle F, P_{1}^{*} G \right\rangle.
\end{align*}
Next, we introduce the following operators
\begin{equation*}
\mathcal{A} = \frac{P_1 - P_{1}^{*}}{2} = \begin{pmatrix}
-\nabla \Phi\cdot \nabla -\frac{1}{2}\Delta \Phi & 0 \\[3mm] 
0 & \frac{1}{2} \partial_{\nu}\Phi - \left\langle \nabla_{\Gamma} \Phi, \nabla_{\Gamma} \right\rangle_{\Gamma} - \frac{1}{2} \Delta_{\Gamma}\Phi
\end{pmatrix},
\end{equation*}
which is antisymmetric on $\mathbb{H}^1$,  and 
\begin{equation*}
\mathcal{S} = \frac{P_1 + P_{1}^{*}}{2} = \begin{pmatrix}
\Delta + \eta & 0 \\[3mm]
- \partial_{\nu} & \Delta_{\Gamma}+ \theta
\end{pmatrix},\\
\end{equation*}
that is symmetric on $\mathbb{H}^1$, where
\begin{align*}
\eta &= \frac{1}{2} \left( \partial_{t}\Phi + \frac{1}{2}\left\lvert\nabla \Phi \right\lvert^{2}\right), \\
\theta &= \frac{1}{2} \left( \partial_{t}\Phi + \frac{1}{2}\left\lvert\nabla_{\Gamma} \Phi \right\lvert^{2}\right).
\end{align*}
Thus
\begin{equation*}
\partial_{t} F = \mathcal{S} F + \mathcal{A} F.
\end{equation*}
\smallskip

\noindent\textbf{Step 2. Energy estimates.} Multiplying the above equation by $F$, we obtain
\begin{equation*}
\frac{1}{2} \partial_{t} \|F\|^{2} - \left\langle \mathcal{S}F, F \right\rangle =0.
\end{equation*}
Introduce the frequency function $\displaystyle 
\mathcal{N}=\frac{\langle-\mathcal{S} F, F\rangle}{\|F\|^{2}}.$ Then
\begin{equation*}
\frac{1}{2} \partial_{t} \|F\|^{2} + \mathcal{N} \|F\|^{2}  =0.
\end{equation*}
We prove here that the derivative of $\mathcal{N}$ satisfies 
\begin{equation*}
	\begin{array}[c]{ll}
		\dfrac{d}{d t} \mathcal{N} & \leq \dfrac{1}{\|F\|^{2}}\bigg( \left\langle-\left(\mathcal{S}^{\prime}+[\mathcal{S}, \mathcal{A}]\right) F, F\right\rangle-\displaystyle \int_{\Gamma} \partial_{\nu}f\left( \mathcal{A}_{1} f\right)_{\Gamma} \mathrm{d} S +\int_{\Gamma} \partial_{\nu} f  \mathcal{A}_{2} f_{\Gamma} \mathrm{d} S \\
		&\quad + \displaystyle\int_{\Gamma} \partial_{\nu} \Phi f_{\Gamma}  \left(\mathcal{S}_{1} f\right)_{\Gamma} \mathrm{d}S - \int_{\Gamma} \partial_{\nu} \Phi f_{\Gamma}  \mathcal{S}_{2} F \mathrm{d} S \bigg), 
	\end{array}
\end{equation*}
where
$ \mathcal{S}F =\left(\begin{array}{l} \mathcal{S}_{1} f \\  \mathcal{S}_{2}\left( f, f_{\Gamma}\right)\end{array}\right)$, $ \mathcal{A}F =\left(\begin{array}{l} \mathcal{A}_{1} f \\  \mathcal{A}_{2} f_{\Gamma}\end{array}\right)$ and $\left\langle [\mathcal{S}, \mathcal{A}] F, F\right\rangle = \langle\mathcal{S}  \mathcal{A} F, F\rangle-\langle\mathcal{A} \mathcal{S} F, F\rangle$. \\
Indeed, we have 
\begin{align}\label{2.17}
\frac{\d}{\d t} \mathcal{N} &=\frac{1}{\|F\|^{4}}\left(\frac{\d}{\d t}\langle-\mathcal{S} F, F\rangle\|F\|^{2}+\langle\mathcal{S} F, F\rangle \frac{\d}{\d t}\|F\|^{2}\right)\nonumber \\
&=\frac{1}{\|F\|^{2}}\left[\left\langle-\mathcal{S}^{\prime} F, F\right\rangle-2\left\langle\mathcal{S} F, F^{\prime}\right\rangle\right]+\frac{2}{\|F\|^{4}}\langle\mathcal{S} F, F\rangle^{2}\nonumber \\
&=\frac{1}{\|F\|^{2}}\left[\left\langle-\mathcal{S}^{\prime} F, F\right\rangle-2\langle\mathcal{S} F, \mathcal{A} F\rangle\right]+\frac{2}{\|F\|^{4}}\left[-\|\mathcal{S} F\|^{2}\|F\|^{2}+\langle\mathcal{S} F, F\rangle^{2}\right]\nonumber\\
&\leq \frac{1}{\|F\|^{2}}\left[\left\langle-\mathcal{S}^{\prime} F, F\right\rangle-2\langle\mathcal{S} F, \mathcal{A} F\rangle\right].
\end{align}
On the other hand, we have
\[
\begin{array}[c]{ll}
\langle\mathcal{S} F, \mathcal{A} F\rangle &= \displaystyle\int_{\Omega} \left( \Delta f + \eta f\right) \mathcal{A}_{1}f \mathrm{d} x + \int_{\Gamma} \left( - \partial_{\nu} f + \Delta_{\Gamma} f_{\Gamma}+\theta f_{\Gamma}\right) \mathcal{A}_{2}f_{\Gamma}\mathrm{d}S\\
&=\displaystyle \int_{\Gamma} \partial_{\nu} f \left(\mathcal{A}_{1}f\right)_{\Gamma} \mathrm{d} S - \int_{\Omega} \nabla f \cdot\nabla\left(\mathcal{A}_{1}f\right) \mathrm{d}x + \int_{\Omega}\eta f \mathcal{A}_{1}f \d x-\displaystyle\int_{\Gamma} \partial_{\nu} f \mathcal{A}_{2} f_{\Gamma} \mathrm{d} S \\
& -\displaystyle\int_{\Gamma} \left\langle\nabla_{\Gamma}f_{\Gamma} , \nabla_{\Gamma}\mathcal{A}_{2} f_{\Gamma}\right\rangle \mathrm{d} S + \int_{\Gamma}\theta f_{\Gamma} \mathcal{A}_{2}f_{\Gamma} \mathrm{d} S\\
&=\displaystyle\int_{\Gamma} \partial_{\nu} f \left(\mathcal{A}_{1}f\right)_{\Gamma} \mathrm{d} S - \int_{\Gamma} f_{\Gamma} \partial_{\nu}\left(\mathcal{A}_{1}f\right) \mathrm{d}S+\int_{\Omega} f \Delta\left(\mathcal{A}_{1}f\right) \mathrm{d}x + \int_{\Omega}\eta f \mathcal{A}_{1}f \mathrm{d}x \\
&-\displaystyle\int_{\Gamma} \partial_{\nu} f \mathcal{A}_{2} f_{\Gamma} \mathrm{d} S + \int_{\Gamma} f_{\Gamma} \Delta_{\Gamma}\left(\mathcal{A}_{2} f_{\Gamma}\right) \mathrm{d} S + \int_{\Gamma}\theta f_{\Gamma} \mathcal{A}_{2}f_{\Gamma} \mathrm{d}\\
&= \displaystyle\int_{\Omega}\left( \Delta + \eta \right)  \left(\mathcal{A}_{1}f\right) f \mathrm{d}x + \int_{\Gamma} \left( -\partial_{\nu} \left( \mathcal{A}_{1} f \right)+ \left( \Delta_{\Gamma}+\theta\right)\left( \mathcal{A}_{2} f_{\Gamma} \right)\right) f_{\Gamma} \mathrm{d} S\\
&+\displaystyle\int_{\Gamma} \partial_{\nu} f \left(\mathcal{A}_{1}f\right)_{\Gamma} \mathrm{d} S -\int_{\Gamma} \partial_{\nu} f \mathcal{A}_{2} f_{\Gamma} \mathrm{d} S.
\end{array}
\]
Therefore, we obtain
\begin{equation}\label{3.15}
\langle\mathcal{S} F, \mathcal{A} F\rangle = \langle\mathcal{S}  \mathcal{A} F, F\rangle + \int_{\Gamma} \partial_{\nu}f  \left[\left( \mathcal{A}_{1} f \right)_{\Gamma}-  \mathcal{A}_{2} f_{\Gamma} \right]\mathrm{d} S.
\end{equation}
Similarly, we prove that 
\begin{equation}\label{3.16}
\langle\mathcal{S} F, \mathcal{A} F\rangle = - \langle\mathcal{A} \mathcal{S}   F, F\rangle - \int_{\Gamma} \partial_{\nu}\Phi f_{\Gamma} \left( \mathcal{S}_{1} f \right)_{\Gamma}\mathrm{d} S + \int_{\Gamma}  \partial_{\nu}\Phi f_{\Gamma}  \mathcal{S}_{2}F \mathrm{d} S.
\end{equation}
By combining \eqref{3.15} and \eqref{3.16}, we obtain
\begin{align}\label{2.20}
 2 \langle\mathcal{S} F, \mathcal{A} F\rangle =& \langle\mathcal{S}  \mathcal{A} F, F\rangle-\langle\mathcal{A}  \mathcal{S} F, F\rangle + \int_{\Gamma} \partial_{\nu}f  \left( \mathcal{A}_{1} f \right)_{\Gamma} \mathrm{d} S- \int_{\Gamma} \partial_{\nu}f  \mathcal{A}_{2} f_{\Gamma}\mathrm{d} S \nonumber\\
 &- \int_{\Gamma} \partial_{\nu}\Phi f_{\Gamma} \left( \mathcal{S}_{1} f \right)_{\Gamma}\mathrm{d} S + \int_{\Gamma}  \partial_{\nu}\Phi f_{\Gamma}  \mathcal{S}_{2}F \mathrm{d} S.
 \end{align}
Combining the equalities \eqref{2.17} and \eqref{2.20} yields the desired formula. 
\bigskip

\noindent\textbf{Step 3.\, Calculating the Carleman commutator.} We show  that
\begin{align}\label{9.1}
& \left \langle -(\mathcal{S}^\prime+[\mathcal{S},A])F,F\right \rangle-\int_{\Gamma}\partial_{\nu}f (A_{1}f)_{|\Gamma}\mathrm{d}S+\int_{\Gamma}\partial_{\nu}f A_{2}f_{\Gamma}\mathrm{d}S\nonumber\\
&+\int_{\Gamma}\partial_{\nu}\Phi f_{\Gamma}(\mathcal{S}_{1}f)_{\Gamma}\mathrm{d} S-\int_{\Gamma}\partial_{\nu}\Phi f_{\Gamma}\mathcal{S}_{2}F\mathrm{d}S\nonumber\\
&=\dfrac{-s}{\Upsilon^3}\int_{\Omega}\left( \varphi+\dfrac{s}{2}|\nabla \varphi|^2\right) |f|^2\mathrm{d}x+\dfrac{s}{\Upsilon}\int_{\Omega}|\nabla f|^2\mathrm{d}x+\dfrac{s}{\Upsilon}\int_{\Gamma}\partial_{\nu}\varphi |\partial_{\nu} f|^2\mathrm{d}x\nonumber\\
&-\dfrac{s^2(2-s)}{4\Upsilon^3}\int_{\Omega}|\nabla\varphi|^2 |f|^2\mathrm{d}x-\dfrac{s}{\Upsilon^3}\int_{\Gamma}\left( \varphi+\dfrac{s}{2}|\nabla_{\Gamma} \varphi|^2\right) |f_{\Gamma}|^2\mathrm{d}S\nonumber\\
&-\dfrac{2s}{\Upsilon}\int_{\Gamma}\nabla_{\Gamma}^2\varphi(\nabla_{\Gamma} f_{\Gamma},\nabla_{\Gamma} f_{\Gamma})\mathrm{d}S+\dfrac{s}{\Upsilon}\int_{\Gamma}(\Delta \varphi+\partial_{\nu}\varphi-\Delta_{\Gamma}\varphi)\partial_{\nu}f f_{\Gamma}\mathrm{d}S\nonumber\\
&+\dfrac{s}{2\Upsilon}\int_{\Gamma}(\Delta_{\Gamma}^2\varphi-\Delta_{\Gamma}\partial_{\nu}\varphi)|f_{\Gamma}|^2\mathrm{d}S-\dfrac{s^2}{2\Upsilon^3}\int_{\Gamma}\left(|\nabla_{\Gamma}\varphi|^2+s\nabla_{\Gamma}^2\varphi(\nabla_{\Gamma}\varphi,\nabla_{\Gamma}\varphi)\right) |f_{\Gamma}|^2\mathrm{d}S\nonumber\\
& +\dfrac{s^3}{4\Upsilon^3}\int_{\Gamma}(\partial_{\nu}\varphi)^3|f_{\Gamma}|^2\mathrm{d}S.
\end{align}
Indeed, one has
\begin{equation*}
\mathcal{S}\mathcal{A}F =  \begin{pmatrix}
\Delta\left(-\nabla \Phi \cdot\nabla f - \frac{1}{2}\Delta\Phi f \right)-\eta\nabla \Phi\cdot\nabla f -\frac{1}{2}\eta\Delta\Phi f \\[3mm]
\partial_{\nu}\left( \nabla\Phi\cdot\nabla f + \frac{1}{2}\Delta\Phi f \right) + \frac{1}{2}\Delta_{\Gamma}\partial_{\nu}\Phi f_{\Gamma} + \langle\nabla_{\Gamma}\partial_{\nu}\Phi, \nabla_{\Gamma}f_{\Gamma}\rangle_{\Gamma}+\frac{1}{2}\partial_{\nu}\Phi\Delta_{\Gamma}f_{\Gamma}\\
-\Delta_{\Gamma}\left(\langle\nabla_{\Gamma}\Phi, \nabla_{\Gamma}f_{\Gamma}\rangle_{\Gamma} \right)-\frac{1}{2}\Delta_{\Gamma}^{2}\Phi f_{\Gamma}-\langle\nabla_{\Gamma}\Delta_{\Gamma}\Phi, \nabla_{\Gamma}f_{\Gamma}\rangle_{\Gamma} -\frac{1}{2}\Delta_{\Gamma}\Phi\Delta_{\Gamma}f\\
+\frac{1}{2}\theta \partial_{\nu}\Phi f_{\Gamma}-\theta\langle\nabla_{\Gamma}\Phi, \nabla_{\Gamma}f_{\Gamma}\rangle_{\Gamma}-\frac{1}{2}\theta\Delta_{\Gamma}\Phi f_{\Gamma}
\end{pmatrix},
\end{equation*}
and
\begin{equation*}
\mathcal{A}\mathcal{S}F =  \begin{pmatrix}
-\nabla\Phi\cdot\nabla\Delta f -\nabla\Phi\cdot\nabla \eta f -\eta\nabla \Phi\cdot\nabla f-\frac{1}{2}\Delta\Phi\Delta f -\frac{1}{2}\eta\Delta\Phi f& \\[3mm]
-\frac{1}{2}\partial_{\nu}\Phi \partial_{\nu} f+ \frac{1}{2}\partial_{\nu}\Phi \Delta_{\Gamma} f_{\Gamma}+\frac{1}{2}\theta\partial_{\nu}\Phi  f_{\Gamma}+\langle\nabla_{\Gamma}\Phi,\nabla_{\Gamma}\partial_{\nu} f\rangle_{\Gamma} &\\
 -\langle\nabla_{\Gamma}\Phi,\nabla_{\Gamma}\Delta_{\Gamma} f_{\Gamma}\rangle_{\Gamma}- \langle\nabla_{\Gamma}\Phi, \nabla_{\Gamma}\theta\rangle_{\Gamma}f_{\Gamma} -\theta\langle\nabla_{\Gamma}\Phi, \nabla_{\Gamma}f_{\Gamma}\rangle_{\Gamma}&\\
 +\frac{1}{2}\Delta_{\Gamma} \Phi\partial_{\nu} f -\frac{1}{2}\Delta_{\Gamma}\Phi \Delta_{\Gamma}f_{\Gamma}-\frac{1}{2}\theta\Delta_{\Gamma}\Phi f_{\Gamma}
\end{pmatrix}.
\end{equation*}
We also have
\begin{align*}
        \mathcal{S}^{\prime}F &=\left(\mathcal{S}F\right)^{\prime} - \mathcal{S} F^{\prime} \notag\\
&= \begin{pmatrix}
\Delta\left(\partial_{t}f\right)+\partial_{t}\eta f+\eta \partial_{t} f \\  -\partial_{t}\left(\partial_{\nu}f\right) +\Delta_{\Gamma}\partial_{t} f_{\Gamma}+\partial_{t}\theta f_{\Gamma}+\theta \partial_{t}f_{\Gamma}
\end{pmatrix} - \begin{pmatrix}
\Delta\left(\partial_{t}f\right)+\eta \partial_{t} f \\  -\partial_{t}\left(\partial_{\nu}f\right) +\Delta_{\Gamma}\partial_{t} f_{\Gamma}+\theta \partial_{t}f_{\Gamma}
\end{pmatrix} \notag\\
&= \begin{pmatrix}
\partial_{t}\eta f \\ \partial_{t}\theta f_{\Gamma}
\end{pmatrix}.
\end{align*}
Hence
\begin{equation}\label{2.25}
\begin{aligned}
-\left(\mathcal{S}^{\prime}+[\mathcal{S}, \mathcal{A}]\right) F = \begin{pmatrix}
-\partial_{t}\eta f +\Delta\left(\nabla \Phi\cdot \nabla f + \frac{1}{2}\Delta\Phi f \right)-\nabla\Phi\cdot\nabla\Delta f -\nabla\Phi\cdot\nabla\eta f -\frac{1}{2}\Delta\Phi\Delta f \\[3mm]
-\partial_{t}\theta f_{\Gamma}-\partial_{\nu}\left( \nabla\Phi\cdot\nabla f + \frac{1}{2}\Delta\Phi f \right) - \frac{1}{2}\Delta_{\Gamma}\partial_{\nu}\Phi f_{\Gamma} - \langle\nabla_{\Gamma}\partial_{\nu}\Phi, \nabla_{\Gamma}f_{\Gamma}\rangle_{\Gamma}\\
+\Delta_{\Gamma}\left(\langle\nabla_{\Gamma}\Phi, \nabla_{\Gamma}f_{\Gamma}\rangle_{\Gamma} \right)+\frac{1}{2}\Delta_{\Gamma}^{2}\Phi f_{\Gamma}+\langle\nabla_{\Gamma}\Delta_{\Gamma}\Phi, \nabla_{\Gamma}f_{\Gamma}\rangle_{\Gamma} -\frac{1}{2}\partial_{\nu}\Phi\partial_{\nu}f\\
+\langle\nabla_{\Gamma}\Phi,\nabla_{\Gamma}\partial_{\nu} f\rangle_{\Gamma} -\langle\nabla_{\Gamma}\Phi,\nabla_{\Gamma}\Delta_{\Gamma} f_{\Gamma}\rangle_{\Gamma}- \langle\nabla_{\Gamma}\Phi, \nabla_{\Gamma}\theta\rangle_{\Gamma}f_{\Gamma} +\frac{1}{2} \Delta_{\Gamma}\Phi\partial_{\nu} f_{\Gamma}
\end{pmatrix}.
\end{aligned}
\end{equation}
On the other hand,
\begin{align}
\label{2.26}
&-\displaystyle\int_{\Gamma}\partial_{\nu}f (A_{1}f)_{|\Gamma}\mathrm{d}S+\int_{\Gamma}\partial_{\nu}f A_{2}f_{\Gamma}\mathrm{d}S+\int_{\Gamma}\partial_{\nu}\Phi f_{\Gamma}(\mathcal{S}_{1}f)_{\Gamma}\mathrm{d} S-\int_{\Gamma}\partial_{\nu}\Phi f_{\Gamma}\mathcal{S}_{2}F\mathrm{d}S\nonumber \\
&=-\displaystyle\int_{\Gamma}\partial_{\nu}f\left(-\nabla \Phi\cdot \nabla f-\frac{1}{2}\Delta \Phi f_{\Gamma}\right)\mathrm{d}S+\int_{\Gamma}\partial_{\nu}f\left(\dfrac{1}{2}\partial_{\nu}\Phi f_{\Gamma}-\left \langle\nabla_{\Gamma}\Phi,\nabla_{\Gamma}f_{\Gamma}\right \rangle_{\Gamma}-\frac{1}{2}\Delta_{\Gamma}\Phi f_{\Gamma}\right)\mathrm{d}S\nonumber \\
& +\displaystyle\int_{\Gamma}\partial_{\nu}\Phi f_{\Gamma}\left((\Delta f)_{|\Gamma}+\eta f_{\Gamma}\right)\mathrm{d}S-\int_{\Gamma}\partial_{\nu}\Phi f_{\Gamma}\left(-\partial_{\nu}f+\Delta_{\Gamma}f_{\Gamma}+\theta f_{\Gamma}\right)\mathrm{d}S \nonumber \\
&=\displaystyle\int_{\Gamma}\partial_{\nu}f \left(\left\langle \nabla_{\Gamma}\Phi,\nabla_{\Gamma}f_{\Gamma}\right\rangle_{\Gamma}+\partial_{\nu}\Phi \partial_{\nu} f\right)\mathrm{d}S+\dfrac{1}{2}\int_{\Gamma}\Delta \Phi \partial_{\nu}f f_{\Gamma}\mathrm{d}S +\dfrac{1}{2}\int_{\Gamma}\partial_{\nu}\Phi \partial_{\nu}f f_{\Gamma}\mathrm{d}S \nonumber\\
&-\displaystyle\int_{\Gamma}\partial_{\nu}f \left\langle\nabla_{\Gamma}\Phi,\nabla_{\Gamma}f_{\Gamma}\right\rangle_{\Gamma}\mathrm{d}S-\dfrac{1}{2}\int_{\Gamma}\partial_{\nu}f \Delta_{\Gamma}\Phi f_{\Gamma}\mathrm{d}S+\int_{\Gamma}\partial_{\nu}\Phi f_{\Gamma}(\Delta f)_{|\Gamma} \mathrm{d}S\nonumber\\
&+\displaystyle\int_{\Gamma}\eta \partial_{\nu}\Phi |f_{\Gamma}|^2\mathrm{d}S+\int_{\Gamma}\partial_{\nu}\Phi \partial_{\nu}f f_{\Gamma} \mathrm{d}S-\int_{\Gamma}\partial_{\nu}\Phi f_{\Gamma}\Delta_{\Gamma}f_{\Gamma}\mathrm{d}S-\int_{\Gamma}\theta \partial_{\nu}\Phi |f_{\Gamma}|^2 \mathrm{d}S\nonumber\\
&=\displaystyle\int_{\Gamma} \partial_{\nu}\Phi |\partial_{\nu}f|^2 \mathrm{d}S+\dfrac{1}{2}\int_{\Gamma}(\Delta\Phi-\Delta_{\Gamma}\Phi)\partial_{\nu}f f_{\Gamma}\mathrm{d}S+\dfrac{3}{2}\int_{\Gamma}\partial_{\nu}\Phi \partial_{\nu}f f_{\Gamma}\mathrm{d}S\nonumber\\
&+\displaystyle\int_{\Gamma}\partial_{\nu}\Phi f_{\Gamma}(\Delta f)_{|\Gamma}\mathrm{d}S-\int_{\Gamma}\partial_{\nu}\Phi f_{\Gamma}\Delta_{\Gamma}f_{\Gamma}\mathrm{d}S+\int_{\Gamma}\partial_{\nu}\Phi (\eta-\theta)|f_{\Gamma}|^2\mathrm{d}S\nonumber\\
&=\displaystyle\int_{\Gamma} \partial_{\nu}\Phi |\partial_{\nu}f|^2 \mathrm{d}S+\dfrac{1}{2}\int_{\Gamma}(\Delta\Phi-\Delta_{\Gamma}\Phi)\partial_{\nu}f f_{\Gamma}\mathrm{d}S+\dfrac{3}{2}\int_{\Gamma}\partial_{\nu}\Phi \partial_{\nu}f f_{\Gamma}\mathrm{d}S\nonumber\\
& +\displaystyle\int_{\Gamma}\partial_{\nu}\Phi f_{\Gamma}(\Delta f)_{|\Gamma}\mathrm{d}S-\int_{\Gamma}\partial_{\nu}\Phi f_{\Gamma}\Delta_{\Gamma}f_{\Gamma}\mathrm{d}S+\dfrac{1}{4}\int_{\Gamma}(\partial_{\nu}\Phi)^3 |f_{\Gamma}|^2 \mathrm{d}S.
\end{align}
Combining \eqref{2.25} and \eqref{2.26}, we obtain
\begin{align*}
& \left\langle -(\mathcal{S}^\prime+[S,A])F,F\right\rangle-\int_{\Gamma}\partial_{\nu}f (A_{1}f)_{|\Gamma}\mathrm{d}S+\int_{\Gamma}\partial_{\nu}f A_{2}f_{\Gamma}\mathrm{d}S\nonumber\\
&+\int_{\Gamma}\partial_{\nu}\Phi f_{\Gamma}(\mathcal{S}_{1}f)_{\Gamma}\mathrm{d} S-\int_{\Gamma}\partial_{\nu}\Phi f_{\Gamma}\mathcal{S}_{2}F\mathrm{d}S\nonumber\\
&=\int_{\Omega}-\partial_{t}\eta |f|^2\mathrm{d}x+\int_{\Omega}\Delta\left( \nabla \Phi\cdot\nabla f+\dfrac{1}{2}\Delta \Phi f\right)f\mathrm{d}x-\int_{\Omega}\nabla\Phi\cdot\nabla\Delta f f\mathrm{d}x\\
&-\int_{\Omega}\nabla\Phi\cdot\nabla \eta |f|^2\mathrm{d}x-\dfrac{1}{2}\int_{\Omega}\Delta \Phi \Delta f f\mathrm{d}x-\int_{\Gamma}\partial_{t}\theta |f_{\Gamma}|^2\mathrm{d}S\\
&-\int_{\Gamma}\partial_{\nu}\left( \nabla \Phi\cdot\nabla f+\dfrac{1}{2}\Delta \Phi f\right)f_{\Gamma}\mathrm{d}S-\dfrac{1}{2}\int_{\Gamma}\Delta_{\Gamma}\partial_{\nu}\Phi|f_{\Gamma}|^2\mathrm{d}S\\
&-\displaystyle\int_{\Gamma}\left\langle\nabla_{\Gamma}\partial_{\nu}\Phi,\nabla_{\Gamma}f_{\Gamma}\right\rangle_{\Gamma}f_{\Gamma}\mathrm{d}S+\int_{\Gamma}\Delta_{\Gamma}\left( \left\langle \nabla_{\Gamma} \Phi,\nabla_{\Gamma}f_{\Gamma}\right\rangle_{\Gamma}\right)f_{\Gamma}\mathrm{d}S\\
&+\dfrac{1}{2}\int_{\Gamma}\Delta_{\Gamma}^2 \Phi |f_{\Gamma}|^2\mathrm{d}S+\int_{\Gamma}\left\langle\nabla_{\Gamma}(\Delta_{\Gamma}\Phi),\nabla_{\Gamma}f_{\Gamma}\right\rangle_{\Gamma}f_{\Gamma}\mathrm{d}S-\dfrac{1}{2}\int_{\Gamma}\partial_{\nu}\Phi \partial_{\nu}f f_{\Gamma}\mathrm{d}S\\
&+\int_{\Gamma}\left\langle\nabla_{\Gamma}\Phi,\nabla_{\Gamma}\partial_{\nu}f\right\rangle_{\Gamma}f_{\Gamma}\mathrm{d}S-\int_{\Gamma}\left\langle \nabla_{\Gamma}\Phi,\nabla_{\Gamma}\Delta_{\Gamma}f_{\Gamma}\right\rangle f_{\Gamma}\mathrm{d}S\\
&-\int_{\Gamma}\left\langle\nabla_{\Gamma}\Phi,\nabla_{\Gamma}\theta\right\rangle_{\Gamma}|f_{\Gamma}|^2\mathrm{d}S+\dfrac{1}{2}\int_{\Gamma}\Delta_{\Gamma}\Phi\partial_{\nu}f f_{\Gamma}\mathrm{d}S+\int_{\Gamma} \partial_{\nu}\Phi |\partial_{\nu}f|^2 \mathrm{d}S\\
&+\dfrac{1}{2}\int_{\Gamma}(\Delta\Phi-\Delta_{\Gamma}\Phi)\partial_{\nu}f f_{\Gamma}\mathrm{d}S+\dfrac{3}{2}\int_{\Gamma}\partial_{\nu}\Phi \partial_{\nu}f f_{\Gamma}\mathrm{d}S\\
&+\int_{\Gamma}\partial_{\nu}\Phi f_{\Gamma}(\Delta f)_{|\Gamma}\mathrm{d}S-\int_{\Gamma}\partial_{\nu}\Phi f_{\Gamma}\Delta_{\Gamma}f_{\Gamma}\mathrm{d}S+\dfrac{1}{4}\int_{\Gamma}(\partial_{\nu}\Phi)^3 |f_{\Gamma}|^2\mathrm{d}S.
\end{align*}
Using integration by parts, we infer that
\begin{align*}
& \left\langle -(\mathcal{S}^\prime+[S,A])F,F\right\rangle-\int_{\Gamma}\partial_{\nu}f (A_{1}f)_{|\Gamma}\mathrm{d}S+\int_{\Gamma}\partial_{\nu}f A_{2}f_{\Gamma}\mathrm{d}S\nonumber\\
&+\int_{\Gamma}\partial_{\nu}\Phi f_{\Gamma}(\mathcal{S}_{1}f)_{\Gamma}\mathrm{d} S-\int_{\Gamma}\partial_{\nu}\Phi f_{\Gamma}\mathcal{S}_{2}F\mathrm{d}S\nonumber\\
&=\int_{\Omega}-\partial_{t}\eta |f|^2\mathrm{d}x+\int_{\Gamma}\partial_{\nu}\left( \nabla \Phi\cdot\nabla f+\dfrac{1}{2}\Delta \Phi f\right)f_{\Gamma}\mathrm{d}S\\
&-\int_{\Omega}\nabla\left( \nabla \Phi\cdot\nabla f+\dfrac{1}{2}\Delta \Phi f\right)\cdot\nabla f\mathrm{d}x-\int_{\Gamma}\partial_{\nu}\Phi (\Delta f)_{|\Gamma}f_{\Gamma}\mathrm{d}S \\
&+\int_{\Omega}\Delta \Phi \Delta f f \mathrm{d}x+\int_{\Omega}\Delta f \nabla\Phi\cdot\nabla f\mathrm{d}x-\int_{\Omega}\nabla \Phi\cdot\nabla \eta |f|^2\mathrm{d}x\\
&-\dfrac{1}{2}\int_{\Gamma}\Delta \Phi \partial_{\nu}f f_{\Gamma}\mathrm{d}S+\dfrac{1}{2}\int_{\Omega}\nabla f\cdot\nabla(\Delta \Phi f)\mathrm{d}x-\int_{\Gamma}\partial_{t}\theta |f_{\Gamma}|^2\mathrm{d}S\\
&-\int_{\Gamma}\partial_{\nu}\left( \nabla \Phi\cdot\nabla f+\dfrac{1}{2}\Delta \Phi f\right)f_{\Gamma}\mathrm{d}S-\dfrac{1}{2}\int_{\Gamma}\Delta_{\Gamma}\partial_{\nu}\Phi|f_{\Gamma}|^2\mathrm{d}S\\
&-\dfrac{1}{2}\int_{\Gamma}\left \langle\nabla_{\Gamma}\partial_{\nu}\Phi,\nabla_{\Gamma}|f_{\Gamma}|^2\right\rangle_{\Gamma}\mathrm{d}S-\int_{\Gamma}\left\langle\nabla_{\Gamma}\left(\left \langle\nabla_{\Gamma}\Phi,\nabla_{\Gamma}f_{\Gamma}\right \rangle_{\Gamma}\right),\nabla_{\Gamma}f_{\Gamma}\right \rangle_{\Gamma}\mathrm{d}S\\
&+\dfrac{1}{2}\int_{\Gamma} \Delta_{\Gamma}^2\Phi|f_{\Gamma}|^2\mathrm{d}S+\dfrac{1}{2}\int_{\Gamma}\left \langle\nabla_{\Gamma}\Delta_{\Gamma}\Phi,\nabla_{\Gamma}|f_{\Gamma}|^2\right \rangle_{\Gamma}\mathrm{d}S+\int_{\Gamma}\partial_{\nu}\Phi \partial_{\nu}f f_{\Gamma}\mathrm{d}S\\
&-\int_{\Gamma}\Delta_{\Gamma}\Phi \partial_{\nu}f f_{\Gamma}\mathrm{d}S-\int_{\Gamma}\left \langle\nabla_{\Gamma}\Phi,\nabla_{\Gamma}f_{\Gamma}\right\rangle_{\Gamma}\partial_{\nu}f\mathrm{d}S+\int_{\Gamma}\Delta_{\Gamma}\Phi \Delta_{\Gamma}f_{\Gamma}f_{\Gamma}\mathrm{d}S\\
&+\int_{\Gamma}\Delta_{\Gamma}f_{\Gamma}\left \langle\nabla_{\Gamma}\Phi,\nabla_{\Gamma}f_{\Gamma}\right \rangle_{\Gamma}\mathrm{d}S-\int_{\Gamma}\left \langle\nabla_{\Gamma}\Phi,\nabla_{\Gamma}\theta\right \rangle_{\Gamma}|f_{\Gamma}|^2\mathrm{d}S+\dfrac{1}{2}\int_{\Gamma}\Delta_{\Gamma}\Phi \partial_{\nu}f f_{\Gamma}\mathrm{d}S\\
&+\int_{\Gamma} \partial_{\nu}\Phi |\partial_{\nu}f|^2 \mathrm{d}S+\dfrac{1}{2}\int_{\Gamma}(\Delta\Phi-\Delta_{\Gamma}\Phi)\partial_{\nu}f f_{\Gamma}\mathrm{d}S\\
&+\displaystyle\int_{\Gamma}\partial_{\nu}\Phi f_{\Gamma}(\Delta f)_{|\Gamma}\mathrm{d}S-\int_{\Gamma}\partial_{\nu}\Phi f_{\Gamma}\Delta_{\Gamma}f_{\Gamma}\mathrm{d}S+\dfrac{1}{4}\int_{\Gamma}(\partial_{\nu}\Phi)^3 |f_{\Gamma}|^2\mathrm{d}S\\
&=-\int_{\Omega}\partial_{t}\eta |f|^2\mathrm{d}x-\int_{\Omega}\nabla^2\Phi (\nabla f,\nabla f)\mathrm{d}x-\dfrac{1}{2}\int_{\Omega}\nabla \Phi\cdot\nabla |\nabla f|^2\mathrm{d}x\\
&+\int_{\Gamma}\Delta \Phi \partial_{\nu}f f_{\Gamma}\mathrm{d}S-\dfrac{1}{2}\int_{\Omega}\nabla \Delta \Phi\cdot\nabla|f|^2\mathrm{d}x-\int_{\Omega}\Delta \Phi |\nabla f|^2\mathrm{d}x\\
&+\int_{\Gamma}\partial_{\nu}f \nabla \Phi\cdot\nabla f\mathrm{d}S-\int_{\Omega}\nabla f\cdot\nabla(\nabla\Phi\cdot\nabla f)\mathrm{d}x-\int_{\Omega}\nabla \Phi\cdot\nabla \eta |f|^2\mathrm{d}x\\
&-\int_{\Gamma}\partial_{t}\theta |f_{\Gamma}|^2\mathrm{d}S-\dfrac{1}{2}\int_{\Gamma}\Delta_{\Gamma}\partial_{\nu}\Phi |f_{\Gamma}|^2\mathrm{d}S+\dfrac{1}{2}\int_{\Gamma}\Delta_{\Gamma}\partial_{\nu}\Phi |f_{\Gamma}|^2\mathrm{d}S\\
&-\int_{\Gamma}\nabla_{\Gamma}^2\Phi(\nabla_{\Gamma}f_{\Gamma},\nabla_{\Gamma}f_{\Gamma})\mathrm{d}S-\dfrac{1}{2}\int_{\Gamma}\left \langle\nabla_{\Gamma}\Phi,\nabla_{\Gamma}|\nabla_{\Gamma}f_{\Gamma}|^2\right \rangle_{\Gamma}\mathrm{d}S\\
&+\dfrac{1}{2}\int_{\Gamma}\Delta_{\Gamma}^2\Phi |f_{\Gamma}|^2\mathrm{d}S-\dfrac{1}{2}\int_{\Gamma}\Delta_{\Gamma}^2\Phi |f_{\Gamma}|^2\mathrm{d}S+\int_{\Gamma}(\partial_{\nu}\Phi-\Delta_{\Gamma}\Phi)\partial_{\nu}f f_{\Gamma}\mathrm{d}S\\
&-\int_{\Gamma}\left \langle\nabla_{\Gamma}\Phi,\nabla_{\Gamma}f_{\Gamma}\right \rangle_{\Gamma}\partial_{\nu}f\mathrm{d}S+\int_{\Gamma}(\Delta_{\Gamma}\Phi-\partial_{\nu}\Phi)f_{\Gamma}\Delta_{\Gamma}f_{\Gamma}\mathrm{d}S\\
&-\int_{\Gamma}\left \langle\nabla_{\Gamma}f_{\Gamma},\nabla_{\Gamma}(\left \langle \nabla_{\Gamma}\Phi,\nabla_{\Gamma}f_{\Gamma}\right \rangle_{\Gamma})\right \rangle_{\Gamma}\mathrm{d}S-\int_{\Gamma}\left \langle\nabla_{\Gamma}\Phi,\nabla_{\Gamma}\theta\right \rangle_{\Gamma} |f_{\Gamma}|^2\mathrm{d}S\\
&+\int_{\Gamma}\partial_{\nu}\Phi |\partial_{\nu}f|^2\mathrm{d}S+\dfrac{1}{4}\int_{\Gamma}(\partial_{\nu}\Phi)^3|f_{\Gamma}|^2\mathrm{d}S\\
&=-\int_{\Omega}\partial_{t}\eta |f|^2\mathrm{d}x-\int_{\Omega}\nabla^2\Phi(\nabla f,\nabla f)\mathrm{d}x-\dfrac{1}{2}\int_{\Gamma}\partial_{\nu}\Phi|\nabla f|^2\mathrm{d}x\\
&+\dfrac{1}{2}\int_{\Omega}\Delta \Phi |\nabla f|^2\mathrm{d}x-\dfrac{1}{2}\int_{\Gamma}\partial_{\nu}\Delta \Phi |f|^2\mathrm{d}S+\dfrac{1}{2}\int_{\Omega}\Delta ^2\Phi |f|^2\mathrm{d}x\\
&-\int_{\Omega}\Delta \Phi |\nabla f|^2\mathrm{d}x+\int_{\Gamma}\partial_{\nu}\Phi |\partial_{\nu}f|^2\mathrm{d}S-\int_{\Omega}\nabla^2\Phi(\nabla f,\nabla f)\mathrm{d}x\\
&-\dfrac{1}{2}\int_{\Omega}\nabla \Phi\cdot\nabla |\nabla f|^2\mathrm{d}x-\int_{\Omega}\nabla \Phi\cdot\nabla\eta |f|^2\mathrm{d}x-\int_{\Gamma}\partial_{t}\theta |f_{\Gamma}|^2\mathrm{d}x\\
&-\int_{\Gamma}\nabla_{\Gamma}^2\Phi (\nabla_{\Gamma}f_{\Gamma},\nabla_{\Gamma}f_{\Gamma})\mathrm{d}S+\dfrac{1}{2}\int_{\Gamma}\Delta_{\Gamma}\Phi|\nabla_{\Gamma}f_{\Gamma}|^2\mathrm{d}x\\
&+\int_{\Gamma}(\Delta \Phi+\partial_{\nu}\Phi-\Delta_{\Gamma}\Phi)\partial_{\nu}f f_{\Gamma}\mathrm{d}S-\int_{\Gamma}\left \langle\nabla_{\Gamma}((\Delta_{\Gamma}\Phi-\partial_{\nu}\Phi)f_{\Gamma}),\nabla_{\Gamma}f_{\Gamma}\right\rangle_{\Gamma}\mathrm{d}S\\
&-\int_{\Gamma}\nabla_{\Gamma}^2\Phi(\nabla_{\Gamma}f_{\Gamma},\nabla_{\Gamma}f_{\Gamma})\mathrm{d}S-\dfrac{1}{2}\int_{\Gamma}\left \langle\nabla_{\Gamma}\Phi,\nabla_{\Gamma}|\nabla_{\Gamma}f_{\Gamma}|^2\right \rangle_{\Gamma}\mathrm{d}S\\
&-\int_{\Gamma}\left \langle\nabla_{\Gamma}\Phi,\nabla_{\Gamma}\theta\right\rangle_{\Gamma}|f_{\Gamma}|^2\mathrm{d}S+\int_{\Gamma}\partial_{\nu}\Phi |\partial_{\nu}f|^2\mathrm{d}S+ \dfrac{1}{4}\int_{\Gamma}(\partial_{\nu}\Phi)^3 |f_{\Gamma}|^2\mathrm{d}S\\
&=-\int_{\Omega}\partial_{t}\eta |f|^2\mathrm{d}x-2\int_{\Omega}\nabla^2\Phi(\nabla f,\nabla f)\mathrm{d}x-\dfrac{1}{2}\int_{\Gamma}\partial_{\nu}\Phi|\nabla f|^2\mathrm{d}x\\
&-\dfrac{1}{2}\int_{\Omega}\Delta \Phi |\nabla f|^2\mathrm{d}x+2\displaystyle\int_{\Gamma}\partial_{\nu}\Phi|\partial_{\nu}f|^2\mathrm{d}S-\dfrac{1}{2}\int_{\Gamma}\partial_{\nu}\Phi|\nabla f|^2\mathrm{d}S\\
&+\dfrac{1}{2}\int_{\Omega}\Delta\Phi |\nabla f|^2\mathrm{d}x-\int_{\Omega}\nabla\Phi\cdot\nabla \eta |f|^2\mathrm{d}x-\int_{\Gamma}\partial_{t}\theta |f_{\Gamma}|^2\mathrm{d}S\\
&-2\int_{\Gamma}\nabla_{\Gamma}^2\Phi(\nabla_{\Gamma} f_{\Gamma},\nabla_{\Gamma} f_{\Gamma})\mathrm{d}S +\dfrac{1}{2}\int_{\Gamma}\Delta_{\Gamma} \Phi |\nabla_{\Gamma} f_{\Gamma}|^2\mathrm{d}S\\
&+\int_{\Gamma}(\Delta \Phi+\partial_{\nu}\Phi-\Delta_{\Gamma}\Phi)\partial_{\nu}f f_{\Gamma}\mathrm{d}S-\dfrac{1}{2}\int_{\Gamma}\left \langle\nabla_{\Gamma}(\Delta_{\Gamma}\Phi-\partial_{\nu}\Phi),\nabla_{\Gamma}|f_{\Gamma}|^2\right \rangle_{\Gamma}\mathrm{d}S\\
&-\int_{\Gamma}(\Delta_{\Gamma}\Phi-\partial_{\nu} \Phi)|\nabla_{\Gamma}f_{\Gamma}|^2\mathrm{d}S+\dfrac{1}{2}\int_{\Gamma}\Delta _{\Gamma}\Phi |\nabla_{\Gamma}f_{\Gamma}|^2\mathrm{d}S-\int_{\Gamma}\left \langle\nabla_{\Gamma}\Phi,\nabla_{\Gamma}\theta \right\rangle_{\Gamma}|f_{\Gamma}|^2\mathrm{d}S\\
&+\dfrac{1}{4}\int_{\Gamma}(\partial_{\nu}\Phi)^3|f_{\Gamma}|^2\mathrm{d}S.
\end{align*}
By the above formula and the fact that $\left\lvert \nabla f\right\lvert^2=\left\lvert \nabla_{\Gamma} f\right\lvert^2+\left\lvert \partial_{\nu} f\right\lvert^2 ,$ we obtain
\begin{align*}
    & \left\langle-(\mathcal{S}^\prime+[S,A])F,F\right\rangle-\int_{\Gamma}\partial_{\nu}f (A_{1}f)_{|\Gamma}\mathrm{d}S+\int_{\Gamma}\partial_{\nu}f A_{2}f_{\Gamma}\mathrm{d}S\nonumber\\
&+\int_{\Gamma}\partial_{\nu}\Phi f_{\Gamma}(\mathcal{S}_{1}f)_{\Gamma}\mathrm{d} S-\int_{\Gamma}\partial_{\nu}\Phi f_{\Gamma}\mathcal{S}_{2}F\mathrm{d}S\nonumber\\
&=-\displaystyle\int_{\Omega}\partial_{t}\eta |f|^2\mathrm{d}x-2\int_{\Omega}\nabla^2\Phi(\nabla f,\nabla f)\mathrm{d}x+\int_{\Gamma}\partial_{\nu}\Phi|\partial_{\nu} f|^2\mathrm{d}S\nonumber\\
&-\int_{\Omega}\nabla\Phi\cdot\nabla \eta |f|^2\mathrm{d}x-\int_{\Gamma}\partial_{t}\theta |f_{\Gamma}|^2\mathrm{d}S-2\int_{\Gamma}\nabla_{\Gamma}^2\Phi(\nabla_{\Gamma} f_{\Gamma},\nabla_{\Gamma} f_{\Gamma})\mathrm{d}S\nonumber\\
&+\int_{\Gamma}(\Delta \Phi+\partial_{\nu}\Phi-\Delta_{\Gamma}\Phi)\partial_{\nu}f f_{\Gamma}\mathrm{d}S+\dfrac{1}{2}\int_{\Gamma}(\Delta_{\Gamma}^2\Phi-\Delta_{\Gamma}\partial_{\nu}\Phi)|f_{\Gamma}|^2\mathrm{d}S\nonumber\\
&-\displaystyle\int_{\Gamma}\left \langle\nabla_{\Gamma}\Phi,\nabla_{\Gamma}\theta \right\rangle_{\Gamma}|f_{\Gamma}|^2\mathrm{d}S+\dfrac{1}{4}\int_{\Gamma}(\partial_{\nu}\Phi)^3|f_{\Gamma}|^2\mathrm{d}S.
\end{align*}
Using the definition of $\Phi\left(x,t\right)= \frac{s \varphi(x)}{\Upsilon(t)},$ we obtain the desired formula \eqref{9.1}.
\smallskip

\noindent\textbf{Step 4.}
For any $h\in(0,1]$ and $s\in (0,1)$ sufficiently small, we prove that
\begin{align}\label{14.11}
&\left\langle-(\mathcal{S}^\prime+[\mathcal{S},A])F,F\right \rangle-\int_{\Gamma}\partial_{\nu}f (A_{1}f)_{|\Gamma}\mathrm{d}S+\int_{\Gamma}\partial_{\nu}f A_{2}f_{\Gamma}\mathrm{d}S \nonumber\\
&+\int_{\Gamma}\partial_{\nu}\Phi f_{\Gamma}(\mathcal{S}_{1}f)_{\Gamma}\mathrm{d} S-\int_{\Gamma}\partial_{\nu}\Phi f_{\Gamma}\mathcal{S}_{2}F\mathrm{d}S\leq \dfrac{1+C_{0}}{\Upsilon}\left\langle-\mathcal{S}F,F\right\rangle+\frac{C}{h^2}\|F\|^2,
\end{align}
where $C=C(\overline{\Omega})> 0$ and $C_{0}=C(s)\in (0,1)$. Indeed, we have
\begin{align*}
\left \langle\mathcal{S}F,F\right \rangle&= \int_{\Omega}(\Delta f +\eta f)f\mathrm{d}x-\int_{\Gamma}\partial_{\nu}f f_{\Gamma}\mathrm{d}S+\int_{\Gamma}(\Delta_{\Gamma}f_{\Gamma}+\theta f_{\Gamma})\mathrm{d}S\nonumber\\
&=\int_{\Omega}\Delta f f\mathrm{d}x+\int_{\Omega}\eta |f|^2\mathrm{d}x-\int_{\Gamma}\partial_{\nu}f f_{\Gamma}\mathrm{d}S+\int_{\Gamma}\Delta_{\Gamma}f_{\Gamma}f_{\Gamma}\mathrm{d}S\nonumber\\
&+\int_{\Gamma}\theta |f_{\Gamma}|^2\mathrm{d}S\nonumber\\
&=\int_{\Gamma}\partial_{\nu}f f_{\Gamma}\mathrm{d}S-\int_{\Omega}|\nabla f|^2\mathrm{d}x+\int_{\Omega}\eta |f|^2\mathrm{d}x-\int_{\Gamma}\partial_{\nu}f f_{\Gamma}\mathrm{d}S\nonumber\\
&-\int_{\Gamma}|\nabla_{\Gamma}f_{\Gamma}|^2\mathrm{d}S+\int_{\Gamma}\theta |f_{\Gamma}|^2\mathrm{d}S\nonumber\\
&=-\int_{\Omega}|\nabla f|^2\mathrm{d}x+\dfrac{s}{2\Upsilon^2}\int_{\Omega}\left(\varphi+\dfrac{s}{2}|\nabla \varphi|^2\right) |f|^2\mathrm{d}x-\int_{\Gamma}|\nabla_{\Gamma} f_{\Gamma}|^2\mathrm{d}S\nonumber\\                    &+\dfrac{s}{2\Upsilon^2}\int_{\Gamma}\left(\varphi+\dfrac{s}{2}|\nabla_{\Gamma} \varphi|^2\right) |f_{\Gamma}|^2\mathrm{d}S.
\end{align*}
Hence
\begin{equation}\label{2.28}
\begin{array}[c]{ll}
\dfrac{1}{\Upsilon}\left\langle-\mathcal{S}F,F\right\rangle &= \dfrac{1}{\Upsilon}\displaystyle\int_{\Omega}|\nabla f|^2\mathrm{d}x+\dfrac{s}{2\Upsilon^3}\int_{\Omega}\left[-\left(\varphi+\dfrac{s}{2}|\nabla \varphi|^2\right)\right] |f|^2\mathrm{d}x+\dfrac{1}{\Upsilon}\int_{\Gamma}|\nabla_{\Gamma} f_{\Gamma}|^2\mathrm{d}S\\
&+\displaystyle\dfrac{s}{2\Upsilon^3}\int_{\Gamma}\left[-\left(\varphi+\dfrac{s}{2}|\nabla_{\Gamma} \varphi|^2\right)\right] |f_{\Gamma}|^2\mathrm{d}S.
\end{array}
\end{equation}
Next, we estimate each term of the formula \eqref{9.1}. For $s\in \left(0,1\right)$ sufficient small we have
\begin{equation}\label{14.1}
        \frac{s}{\Upsilon}\int_{\Omega}|\nabla f|^2\mathrm{d}x-\frac{2s}{\Upsilon}\int_{\Gamma}\nabla_{\Gamma}^2\varphi(\nabla_{\Gamma}f_{\Gamma},\nabla_{\Gamma}f_{\Gamma})\mathrm{d}S\leq \frac{1}{\Upsilon}\int_{\Omega}|\nabla f|^2\mathrm{d}x+\frac{1}{\Upsilon}\int_{\Gamma}|\nabla_{\Gamma}f_{\Gamma}|^2\mathrm{d}S.
\end{equation}
By using $|\nabla\varphi|^2 = -\varphi$, we obtain
\begin{align}\label{14.2}
    &\frac{-s}{\Upsilon^3}\int_{\Omega}\left(\varphi+\frac{s}{2}|\nabla\varphi|^2\right)|f|^2\mathrm{d}x-\frac{s^2(2-s)}{4Y^3}\int_{\Omega}|\nabla\varphi|^2 |f|^2\mathrm{d}x\nonumber \\
    &=\frac{1}{\Upsilon ^3}\left(-s+s^2-\frac{s^3}{4}\right)\int_{\Omega}\varphi |f|^2\mathrm{d}x\nonumber\\
    &=\frac{-s(4+s^2-4s)}{4\Upsilon ^3}\int_{\Omega}\varphi |f|^2\mathrm{d}x\nonumber\\
    &=\frac{-s(2-s)^2}{4\Upsilon^3}\int_{\Omega}\varphi |f|^2\mathrm{d}x,
\end{align}
also,
\begin{equation}\label{14.3}
  \frac{-s}{2\Upsilon^3}\int_{\Omega}\left(\varphi+\frac{s}{2}|\nabla \varphi|^2\right)|f|^2\mathrm{d}x=\frac{-s(2-s)}{4\Upsilon^3}\int_{\Omega}\varphi |f|^2\mathrm{d}x,  
\end{equation}
using \eqref{14.2} and \eqref{14.3}, we obtain
\begin{equation}\label{14}
    \frac{-s}{\Upsilon^3}\int_{\Omega}\left(\varphi+\frac{s}{2}|\nabla\varphi|^2\right)\mathrm{d}x-\frac{s^2(2-s)}{4Y^3}\int_{\Omega}|\nabla\varphi|^2 |f|^2\mathrm{d}x= (2-s)\frac{-s}{2\Upsilon^3}\int_{\Omega}\left(\varphi+\frac{s}{2}|\nabla \varphi|^2\right)|f|^2\mathrm{d}x.
\end{equation}
On the other hand,
\begin{align}\label{15}
    &\frac{-s}{\Upsilon^3}\int_{\Gamma}\left(\varphi+\frac{s}{2}|\nabla_{\Gamma}\varphi|^2\right)|f_{\Gamma}|^2\mathrm{d}S-\frac{s^2}{2\Upsilon^3}\int_{\Gamma}\left(|\nabla_{\Gamma}\varphi|^2+s\nabla_{\Gamma}^2\varphi (\nabla_{\Gamma}\varphi,\nabla_{\Gamma}\varphi)\right)|f_{\Gamma}|^2\mathrm{d}S\nonumber\\
    &+\frac{s^3}{4\Upsilon^3}\int_{\Gamma}(\partial_{\nu}\varphi)^3|f_{\Gamma}|^2\mathrm{d}S\nonumber\\
    &=2\left[\frac{-s}{2\Upsilon^3}\int_{\Gamma}\left(\varphi+\frac{s}{2}|\nabla_{\Gamma}\varphi|^2\right)|f_{\Gamma}|^2\mathrm{d}S\right]+\frac{s^3}{4\Upsilon^3}\int_{\Gamma}(\partial_{\nu}\varphi)^3|f_{\Gamma}|^2\mathrm{d}S\nonumber\\
    &-\frac{s^2}{2\Upsilon^3}\int_{\Gamma}\left(|\nabla_{\Gamma}\varphi|^2+s\nabla_{\Gamma}^2\varphi (\nabla_{\Gamma}\varphi,\nabla_{\Gamma}\varphi)\right)|f_{\Gamma}|^2\mathrm{d}S.
\end{align}
Choosing $s\in (0,1)$ sufficiently small, using $\partial_{\nu}\varphi<0$ (by Lemma \ref{nd}) and the fact that
$$\varphi+\frac{s}{2}|\nabla_{\Gamma}\varphi|^2\le \varphi+|\nabla_{\Gamma}\varphi|^2=-|\nabla \varphi|^2+|\nabla_{\Gamma}\varphi|^2=-|\partial_{\nu}\varphi|^2 <0,$$ we obtain 
\begin{equation}\label{16}
    \dfrac{s^3}{4}(\partial_{\nu}\varphi)^3\leq \frac{s^4}{2}\left(\varphi+\frac{s}{2}|\nabla_{\Gamma}\varphi|^2\right),
\end{equation}
and
\begin{equation}\label{17}
    |\nabla_{\Gamma}\varphi|^2+s\nabla_{\Gamma}^2\varphi(\nabla_\Gamma\varphi,\nabla_{\Gamma}\varphi)\geq 0.
\end{equation}
Therefore, combining \eqref{15}-\eqref{17}, we obtain
\begin{align}\label{18}
    &\frac{-s}{\Upsilon^3}\int_{\Gamma}\left(\varphi+\frac{s}{2}|\nabla_{\Gamma}\varphi|^2\right)|f_{\Gamma}|^2\mathrm{d}S-\frac{s^2}{2\Upsilon^3}\int_{\Gamma}\left(|\nabla_{\Gamma}\varphi|^2+s\nabla_{\Gamma}^2\varphi (\nabla_{\Gamma}\varphi,\nabla_{\Gamma}\varphi)\right)|f_{\Gamma}|^2\mathrm{d}S\nonumber\\
    &+\frac{s^3}{4\Upsilon^3}\int_{\Gamma}(\partial_{\nu}\varphi)^3|f_{\Gamma}|^2\mathrm{d}S\nonumber\\
    &\leq (2-s^3)\left[\frac{-s}{2\Upsilon^3}\int_{\Gamma}\left(\varphi+\frac{s}{2}|\nabla_{\Gamma}\varphi|^2\right)|f_{\Gamma}|^2\mathrm{d}S\right].
\end{align}
Furthermore, the last sum terms of \eqref{9.1} satisfies
\begin{align}\label{2.30}
&\dfrac{s}{\Upsilon}\int_{\Gamma}(\Delta \varphi+\partial_{\nu}\varphi-\Delta_{\Gamma}\varphi)\partial_{\nu}f f_{\Gamma}\mathrm{d}S +\dfrac{s}{2\Upsilon}\int_{\Gamma}(\Delta_{\Gamma}^2\varphi-\Delta_{\Gamma}\partial_{\nu}\varphi)|f_{\Gamma}|^2\mathrm{d}S\nonumber\\
&\leq \frac{s}{2\Upsilon}\int_{\Gamma}|\partial_{\nu}\varphi||\partial_{\nu}f|^2\mathrm{d}S+\frac{C}{h^2}\int_{\Gamma}|f_{\Gamma}|^2\mathrm{d}S,
\end{align}
where $C=C(\overline{\Omega})$ is a positive constant. Indeed, let $\varepsilon > 0$. Using Young inequality, $\partial_{\nu}\varphi<0$ (by Lemma \ref{nd}) and the fact that $\varphi\in C^{\infty}(\overline{\Omega})$. Then there exists a positive constant $C$ such that
\begin{align*}
\dfrac{s}{\Upsilon}\int_{\Gamma}(\Delta \varphi+\partial_{\nu}\varphi-\Delta_{\Gamma}\varphi)\partial_{\nu}f f_{\Gamma}\mathrm{d}S& =      \dfrac{s}{\Upsilon}\int_{\Gamma}\left(\frac{1}{\sqrt{\varepsilon}}|\partial_{\nu}\varphi|^{\frac{-1}{2}}(\Delta \varphi+\partial_{\nu}\varphi-\Delta_{\Gamma}\varphi)f_{\Gamma}\right)\nonumber\\
&\quad\quad\quad\quad\quad\quad\left(\sqrt{\varepsilon}|\partial_{\nu}\varphi|^{\frac{1}{2}}\partial_{\nu}f\right) \mathrm{d}S\nonumber\\
&\leq \dfrac{sC}{\varepsilon \Upsilon}\int_{\Gamma}|f_{\Gamma}|^2\mathrm{d}S+\dfrac{s\varepsilon}{\Upsilon}\int_{\Omega}|\partial_{\nu}\varphi| |\partial_{\nu}f|^2\mathrm{d}S.
\end{align*} 
Furthermore,
\[
\dfrac{s}{2 \Upsilon}\int_{\Gamma}(\Delta_{\Gamma}^2 \varphi-\Delta_{\Gamma}\partial_{\nu}\varphi))|f_{\Gamma}|^2\mathrm{d}S\leq \dfrac{sC}{2 \Upsilon}\int_{\Gamma}|f_{\Gamma}|^2\mathrm{d}S.
\]
Choosing $\varepsilon=\dfrac{1}{2}$ and using the fact that $s\in (0,1)$, $h\in(0,1]$ and $\frac{1}{\Upsilon}\leq \frac{1}{h}\leq \frac{1}{h^2}$ leads us to the desired inequality \eqref{2.30}.
Using \eqref{9.1}, \eqref{14.1}, \eqref{14}, \eqref{18} and \eqref{2.30}, we have
\begin{align}\label{01}
& \left \langle(-\mathcal{S}^\prime+[S,A])F,F\right \rangle-\int_{\Gamma}\partial_{\nu}f (A_{1}f)_{|\Gamma}\mathrm{d}S+\int_{\Gamma}\partial_{\nu}f A_{2}f_{\Gamma}\mathrm{d}S+\int_{\Gamma}\partial_{\nu}\Phi f_{\Gamma}(\mathcal{S}_{1}f)_{\Gamma}\mathrm{d} S\nonumber\\
& -\int_{\Gamma}\partial_{\nu}\Phi f_{\Gamma}\mathcal{S}_{2}F\mathrm{d}S\nonumber\\
&\leq  \dfrac{1}{\Upsilon}\displaystyle\int_{\Omega}|\nabla f|^2\mathrm{d}x+(2-s)\left[\dfrac{-s}{2\Upsilon^3}\int_{\Omega}\left(\varphi+\dfrac{s}{2}|\nabla \varphi|^2\right) |f|^2\mathrm{d}x\right]+\dfrac{1}{\Upsilon}\int_{\Gamma}|\nabla_{\Gamma} f_{\Gamma}|^2\mathrm{d}S\nonumber\\
&+(2-s^3)\left[\dfrac{-s}{2\Upsilon^3}\int_{\Gamma}\left(\varphi+\dfrac{s}{2}|\nabla_{\Gamma} \varphi|^2\right) |f_{\Gamma}|^2\mathrm{d}S\right]+\dfrac{s}{\Upsilon}\int_{\Gamma}\partial_{\nu}\varphi |\partial_{\nu}f|^2 \mathrm{d}S\nonumber\\
&+\dfrac{s}{2\Upsilon}\int_{\Gamma}|\partial_{\nu}\varphi||\partial_{\nu} f|^2\mathrm{d}S+\dfrac{C}{h^2}\|f_{\Gamma}\|_{L^2(\Gamma)}^2.
\end{align}
Since $|\partial_{\nu}\varphi|=-\partial_{\nu}\varphi$,  we obtain
\begin{align}\label{02}
    \dfrac{s}{\Upsilon}\int_{\Gamma}\partial_{\nu}\varphi |\partial_{\nu}f|^2 \mathrm{d}S+\dfrac{s}{2\Upsilon}\int_{\Gamma}|\partial_{\nu}\varphi||\partial_{\nu} f|^2\mathrm{d}S=\dfrac{s}{2\Upsilon}\int_{\Gamma}\partial_{\nu}\varphi|\partial_{\nu} f|^2\mathrm{d}S\leq 0.
\end{align}
By \eqref{01}, \eqref{02} and $\|f_{\Gamma}\|_{L^2(\Gamma)}^2\leq \|F\|^2$, we obtain
\begin{align*}
& \left \langle(-\mathcal{S}^\prime+[S,A])F,F\right \rangle-\int_{\Gamma}\partial_{\nu}f (A_{1}f)_{|\Gamma}\mathrm{d}S+\int_{\Gamma}\partial_{\nu}f A_{2}f_{\Gamma}\mathrm{d}S+\int_{\Gamma}\partial_{\nu}\Phi f_{\Gamma}(\mathcal{S}_{1}f)_{\Gamma}\mathrm{d} S\nonumber\\
& -\int_{\Gamma}\partial_{\nu}\Phi f_{\Gamma}\mathcal{S}_{2}F\mathrm{d}S\nonumber\\
&\leq  \dfrac{1}{\Upsilon}\displaystyle\int_{\Omega}|\nabla f|^2\mathrm{d}x+(2-s)\left[\dfrac{-s}{2\Upsilon^3}\int_{\Omega}\left(\varphi+\dfrac{s}{2}|\nabla \varphi|^2\right) |f|^2\mathrm{d}x\right]+\dfrac{1}{\Upsilon}\int_{\Gamma}|\nabla_{\Gamma} f_{\Gamma}|^2\mathrm{d}S\nonumber\\
&+(2-s^3)\left[\dfrac{-s}{2\Upsilon^3}\int_{\Gamma}\left(\varphi+\dfrac{s}{2}|\nabla_{\Gamma} \varphi|^2\right) |f_{\Gamma}|^2\mathrm{d}S\right]+\dfrac{C}{h^2}\|F\|^2.
\end{align*}
For $s\in (0,1)$, we have $1<2-s<2-s^3$. Hence
\begin{align*}
& \left \langle(-\mathcal{S}^\prime+[S,A])F,F\right \rangle-\int_{\Gamma}\partial_{\nu}f (A_{1}f)_{|\Gamma}\mathrm{d}S+\int_{\Gamma}\partial_{\nu}f A_{2}f_{\Gamma}\mathrm{d}S+\int_{\Gamma}\partial_{\nu}\Phi f_{\Gamma}(\mathcal{S}_{1}f)_{\Gamma}\mathrm{d} S\nonumber\\
& -\int_{\Gamma}\partial_{\nu}\Phi f_{\Gamma}\mathcal{S}_{2}F\mathrm{d}S\nonumber\\
&\leq (2-s^3)\left[ \dfrac{1}{\Upsilon}\displaystyle\int_{\Omega}|\nabla f|^2\mathrm{d}x+\dfrac{s}{2\Upsilon^3}\int_{\Omega}\left[-\left(\varphi+\dfrac{s}{2}|\nabla \varphi|^2\right)\right] |f|^2\mathrm{d}x+\dfrac{1}{\Upsilon}\int_{\Gamma}|\nabla_{\Gamma} f_{\Gamma}|^2\mathrm{d}S\right.\nonumber\\
&+\left.\dfrac{s}{2\Upsilon^3}\int_{\Gamma}\left[-\left(\varphi+\dfrac{s}{2}|\nabla_{\Gamma} \varphi|^2\right)\right] |f_{\Gamma}|^2\mathrm{d}S\right]+\dfrac{C}{h^2}\|F\|^2.
\end{align*}
This leads to the desired inequality \eqref{14.11}, with 
$C_{0}=1-s^3$.
\smallskip

\noindent\textbf{Step 5.} The following differential inequalities hold
\begin{empheq}[left = \empheqlbrace]{alignat=2}
\begin{aligned}
&\frac{1}{2}\frac{\d}{\d t}\left\Vert F\left(  \cdot,t\right)
\right\Vert ^{2}+\mathcal{N}\left(t\right)  \left\Vert F\left(
\cdot,t\right)  \right\Vert ^{2}=0, & \\
&\frac{\d}{\d t}\mathcal{N}\left(  t\right)  \leq\frac
{1+C_{0}}{\Upsilon\left(t\right)  }\mathcal{N}\left(  t\right)+\frac{C}{h^2},
\end{aligned}
\end{empheq}
where $C_{0}=1-s^3$ and $C=C(\overline{\Omega})$ is a positive constant.

Using \cite[Proposition 3]{RBKDP}, we infer, for any $0<t_{1}<t_{2}<t_{3}\leq T$, that
\[
\left(\left\Vert F\left(\cdot,t_{2}\right)  \right\Vert ^{2}\right)
^{1+M}\leq\left(  \left\Vert F\left(  \cdot,t_{1}\right)  \right\Vert
^{2}\right)  ^{M}\left\Vert F\left(  \cdot,t_{3}\right)  \right\Vert ^{2}\mathrm{e}^{D},%
\]
where%
\[
M=\dfrac{\displaystyle\int_{t_{2}}^{t_{3}}\dfrac{1}{(T-t+h)^{1+C_{0}}} \mathrm{d}t }{\displaystyle\int_{t_{1}}^{t_{2}}\dfrac{1}{(T-t+h)^{1+C_{0}}} \mathrm{d}t} \quad \text{and} \quad D=2(1+M)(t_{3}-t_{1})^2 \frac{C}{h^2}. %
\]
Consequently, we obtain
\[
\left( || U\left(  \cdot,t_{2}\right) \mathrm{e}^{\frac{\Phi\left(  \cdot,t_{2}\right)}{2}} || ^{2}  \right)  ^{1+M}\leq\left(
|| U\left(  \cdot,t_{1}\right) \mathrm{e}^{\frac{\Phi\left(  \cdot,t_{1}\right)}{2}} || ^{2}\right)^{M}|| U\left(  \cdot,t_{3}\right) \mathrm{e}^{\frac{\Phi\left(  \cdot,t_{3}\right)}{2}} || ^{2} \mathrm{e}^D\text{ .}%
\]
\smallskip

\noindent\textbf{Step 6.} We take off the weight function $\Phi$ from the integrals
\begin{equation}\label{2.36}
\begin{array}
[c]{ll}%
\left(||U\left(  \cdot,t_{2}\right) ||^{2}  \right)  ^{1+M}&\leq\exp\left[  -\left(  1+M\right)
\min\limits_{x\in\overline{\Omega}}\Phi\left(  x,t_{2}\right)
+M\max\limits_{x\in\overline{\Omega}}\Phi\left(  x,t_{1}\right)
\right]
\\
& \quad \times\left(
|| U\left(  \cdot,t_{1}\right)  || ^{2}\right)  ^{M}|| U\left(  \cdot,t_{3}\right) \mathrm{e}^{\frac{\Phi\left(  \cdot,t_{3}\right)}{2}} ||^{2} \mathrm{e}^D\text{ .}%
\end{array}
\end{equation}
Let $\omega$ be a nonempty open subset of $\Omega$. Then
\[
\begin{array}[c]{ll}
\left\lVert U\left(  \cdot,t_{3}\right) \mathrm{e}^{\frac{\Phi\left(  \cdot,t_{3}\right)}{2}} \right\rVert ^{2}&= \displaystyle\int_{\Omega}\left\vert u\left(  x,t_{3}\right)  \right\vert^{2}\mathrm{e}^{\Phi\left(  x,t_{3}\right)  } \d x+\displaystyle\int_{\Gamma}\left\vert u_{\Gamma}\left(  x,t_{3}\right)  \right\vert
^{2}\mathrm{e}^{\Phi\left(  x,t_{3}\right) }\mathrm{d}S\\
&= \displaystyle\int_{\omega}\left\vert u\left(  x,t_{3}\right)  \right\vert
^{2}\mathrm{e}^{\Phi\left(  x,t_{3}\right)  }\mathrm{d} x+\displaystyle\int_{\left.  \Omega\right\backslash \omega}\left\vert u\left(  x,t_{3}\right)  \right\vert^{2}\mathrm{e}^{\Phi\left(  x,t_{3}\right)  }\mathrm{d}x\\
&\quad+\displaystyle\int_{\Gamma}\left\vert u_{\Gamma}\left(  x,t_{3}\right)  \right\vert
^{2}\mathrm{e}^{\Phi\left(x,t_{3}\right) }\d S\\
& \leq\exp\left[ \max\limits_{x\in\overline{\omega}}
\Phi\left(  x,t_{3}\right)  \right]  \displaystyle\int_{\omega}\left\vert
u\left(  x,t_{3}\right)  \right\vert ^{2}\mathrm{d}x\\
&\quad+\exp\left[\max\limits_{x\in\overline{\left.  \Omega
\right\backslash \omega}}\Phi\left(  x,t_{3}\right)  \right]
\displaystyle\int_{\Omega}\left\vert u\left(  x,t_{3}\right)  \right\vert
^{2}\mathrm{d}x\\
&\quad+\exp\left[\max\limits_{x\in\Gamma}\Phi\left(  x,t_{3}\right)  \right]\displaystyle\int_{\Gamma}\left\vert
u_{\Gamma}\left(  x,t_{3}\right)  \right\vert ^{2}\mathrm{d}S.
\end{array}
\]
Since $\omega$ is a nonempty open subset of $\Omega$, then 
\[\exp\left[  \max\limits_{x\in\Gamma} \Phi\left(  x,t_{3}\right)  \right]\leq\exp\left[\max\limits_{x\in\overline{\left.  \Omega
\right\backslash \omega}}\Phi\left(  x,t_{3}\right)  \right]. \]
Therefore,
\begin{equation}\label{2.37}
\left\lVert U\left(  \cdot,t_{3}\right) \mathrm{e}^{\frac{\Phi\left(  \cdot,t_{3}\right)}{2}} \right\rVert ^{2}\leq \exp\left[\max\limits_{x\in\overline{\omega}}\Phi\left(  x,t_{3}\right)  \right]  \displaystyle\int_{\omega}\left\vert
u\left(  x,t_{3}\right)  \right\vert ^{2}\mathrm{d}x+\exp\left[\max\limits_{x\in\overline{\left.  \Omega\right\backslash \omega}}\Phi\left(  x,t_{3}\right)  \right]|| U(\cdot,t_{3}) ||^2.
\end{equation}
Using \eqref{2.36} and \eqref{2.37}, we obtain
\[
\begin{array}
[c]{ll}%
\left( || U\left(  \cdot,t_{2}\right)  || ^{2}  \right)  ^{1+M}&\leq \mathrm{e}^D\exp\left[  -\left(  1+M\right)
\min\limits_{x\in\overline{\Omega}}\Phi\left(  x,t_{2}\right)
+M\max\limits_{x\in\overline{\Omega}}\Phi\left(  x,t_{1}\right)
+\max\limits_{x\in\overline{\omega}}\Phi\left(  x,t_{1}\right)\right]
\\
& \quad \times\left(
|| U\left(  \cdot,t_{1}\right)  || ^{2}\right)  ^{M}\displaystyle\int_{\omega}\left\vert
u\left(  x,t_{3}\right)  \right\vert ^{2}\mathrm{d}x \\
&+ \mathrm{e}^D\exp\left[  -\left(  1+M\right)
\min\limits_{x\in\overline{\Omega}}\Phi\left(  x,t_{2}\right)
+M\max\limits_{x\in\overline{\Omega}}\Phi\left(  x,t_{1}\right)
+\max\limits_{x\in\overline{\left.  \Omega\right\backslash \omega}}\Phi\left(  x,t_{1}\right)\right]\\
&\quad  \times\left(
|| U\left(  \cdot,t_{1}\right)  || ^{2}\right)  ^{M}||U(\cdot,t_{3})||^2 .
\end{array}
\]
Using the fact that $\left\Vert U\left(  \cdot,T\right)  \right\Vert
\leq\left\Vert U\left(  \cdot,t\right)  \right\Vert \leq\left\Vert U\left(
\cdot,0\right)  \right\Vert ,$ $ 0<t<T$,\, the above inequality becomes%
\[%
\begin{array}
[c]{ll}%
\left(  \left\Vert U\left(  \cdot,T\right)  \right\Vert ^{2}\right)  ^{1+M} &
\leq \mathrm{e}^D\exp\left[  -\left(  1+M\right) \min\limits_{x\in\overline{\Omega}%
} \Phi\left(  x,t_{2}\right)  +M\max\limits_{x\in\overline{\Omega}%
}\Phi\left(  x,t_{1}\right)  +\max\limits_{x\in\overline{\omega}%
}\Phi\left(  x,t_{3}\right)  \right] \\
& \quad\times\left(  \left\Vert U\left(  \cdot,0\right)  \right\Vert
^{2}\right)  ^{M}\displaystyle\int_{\omega}\left\vert u\left(  x,t_{3}\right)
\right\vert ^{2}\mathrm{d}x\\
&  +\mathrm{e}^D\exp\left[  -\left(  1+M\right)  \min\limits_{x\in\overline{\Omega
}}\Phi\left(  x,t_{2}\right)  +M\max\limits_{x\in\overline{\Omega}%
}\Phi\left(  x,t_{1}\right)  +\max\limits_{x\in\overline{\left.
\Omega\right\backslash \omega}}\Phi\left(  x,t_{3}\right)  \right]
\\
& \quad\times\left(  \left\Vert U\left(  \cdot,0\right)  \right\Vert
^{2}\right)^{1+M}.%
\end{array}
\]
Since  $\Phi\left(  x,t\right)  =\displaystyle\frac{s\varphi\left(x\right)  }{T-t+h}$, then
\[%
\begin{array}
[c]{ll}%
\left\Vert U\left(  \cdot,T\right)  \right\Vert ^{1+M} & \leq \mathrm{e}^D\exp%
\frac{s}{2}\left[  -\frac{1+M}{T-t_{2}+h}\min\limits_{x\in\overline{\Omega}%
}\varphi\left(  x\right)  +\frac{M}{T-t_{1}+h}\max\limits_
{x\in\overline{\Omega}}\varphi\left(  x\right)  +\frac{1}%
{T-t_{3}+h}\max\limits_{x\in\overline{\omega}}\varphi\left(
x\right)  \right] \\
& \quad\times\left\Vert U\left(  \cdot,0\right)  \right\Vert ^{M}\left\Vert
u\left(  \cdot,t_{3}\right)  \right\Vert _{L^{2}\left(  \omega\right)  }\\
&+\mathrm{e}^D\exp\frac{s}{2}\left[  -\frac{1+M}{T-t_{2}+h}\min\limits_
{x\in\overline{\Omega}}\varphi\left(  x\right)  +\frac{M}%
{T-t_{1}+h}\max\limits_{x\in\overline{\Omega}}\varphi\left(
x\right)  +\frac{1}{T-t_{3}+h}\max\limits_{x\in\overline{\left.
\Omega\right\backslash \omega}} \varphi\left(  x\right)  \right] \\
& \quad\times\left\Vert U\left(  \cdot,0\right)  \right\Vert ^{1+M}\text{ .}%
\end{array}
\]
\bigskip

\noindent\textbf{Step 7.} We choose $t_{3}=T$, $t_{2}=T-\ell h$, $t_{1}=T-2\ell h$ and $\ell>1$, such that $0<2\ell
h<T.$ Then%
\[%
\begin{array}
[c]{ll}%
\left\Vert u\left(  \cdot,T\right)  \right\Vert ^{1+M_{\ell}} & \leq
 \mathrm{e}^{D_{\ell}}\exp\frac{s}{2h}\left[  -\frac{1+M_{\ell}}{1+\ell}\min\limits_
{x\in\overline{\Omega}}\varphi\left(  x\right)  +\frac{M_{\ell}%
}{1+2\ell}\max\limits_{x\in\overline{\Omega}}\varphi\left(  x\right)
+\max\limits_{x\in\overline{\omega}}\varphi\left(  x\right)  \right]
\\
& \quad\times\left\Vert U\left(  \cdot,0\right)  \right\Vert ^{M_{\ell}%
}\left\Vert u\left(  \cdot,T\right)  \right\Vert _{L^{2}\left(  \omega\right)
}\\
& +\mathrm{e}^{D_{\ell}}\exp\frac{s}{2h}\left[  -\frac{1+M_{\ell}}{1+\ell}%
\min\limits_{x\in\overline{\Omega}}\varphi\left(  x\right)
+\frac{M_{\ell}}{1+2\ell}\max\limits_{x\in\overline{\Omega}}%
\varphi\left(  x\right)  +\max\limits_{x\in\overline{\left.  \Omega
\right\backslash \omega}}\varphi\left(  x\right)  \right] \\
& \quad\times\left\Vert U\left(  \cdot,0\right)  \right\Vert ^{1+M_{\ell}%
}\text{ ,}%
\end{array}
\]
where $M_\ell = \dfrac{(\ell+1)^{C_{0}}-1}{1-\left(\dfrac{\ell+1}{2\ell+1}\right)^{C_{0}}}$ and $D_{\ell}=2C\ell^2(1+M_{\ell})$. \\
Since $\varphi(x)\leq 0$, $\forall\, x\in \overline{\Omega}$, $x_{0}\in \omega$ and $\varphi(x_{0})=0$,  then $\max\limits_{x\in\overline{\Omega}}%
\varphi\left(  x\right)=0$.  
Therefore, 
\[
 -\frac{1+M_{\ell}}{1+\ell}%
 \min\limits_{x\in\overline{\Omega}}\varphi\left(  x\right)
+\frac{M_{\ell}}{1+2\ell}\max\limits_{x\in\overline{\Omega}}
\varphi\left(  x\right)  +\max\limits_{x\in\overline{\left.  \Omega
\right\backslash \omega}}\varphi\left(  x\right)= -\frac{1+M_{\ell}}{1+\ell}%
\min\limits_{x\in\overline{\Omega}}\varphi\left(  x\right)
+\max\limits_{x\in\overline{\left.  \Omega
\right\backslash \omega}}\varphi\left(  x\right).
\]
Now, we use the fact that $C_{0} \in \left(0,1\right)$ and choose $\ell>1$ sufficiently large in order that
\[
 -\frac{1+M_{\ell}}{1+\ell}%
\min\limits_{x\in\overline{\Omega}}\varphi\left(  x\right)
+\frac{M_{\ell}}{1+2\ell}\max\limits_{x\in\overline{\Omega}}%
\varphi\left(  x\right)  +\max\limits_{x\in\overline{\left.  \Omega
\right\backslash \omega}}\varphi\left(  x\right) < 0.
\]
Consequently, there are $C_{1}>0$ and $C_{2}>0$ such that for any $h>0$
with $0<2\ell h<T$,
\[
\left\Vert U\left(  \cdot,T\right)  \right\Vert ^{1+M_{\ell}}\leq
\mathrm{e}^{D_{\ell}} \mathrm{e}^{C_{1}\frac{1}{h}}\left\Vert U\left(  \cdot,0\right)  \right\Vert
^{M_{\ell}}\left\Vert u\left(  \cdot,T\right)  \right\Vert _{L^{2}\left(
\omega\right)  }+ \mathrm{e}^{D_{\ell}} \mathrm{e}^{-C_{2}\frac{1}{h}}\left\Vert U\left(  \cdot,0\right)
\right\Vert ^{1+M_{\ell}}.%
\]
Therefore, we obtain 
\begin{equation}\label{42}
\left\Vert U\left(\cdot,T\right)  \right\Vert ^{1+M_{\ell}}\leq
\mathrm{e}^{D_{\ell}} \mathrm{e}^{C_{1}\frac{1}{h}}\left\Vert U\left(  \cdot,0\right)  \right\Vert
^{M_{\ell}}\left\Vert u\left(  \cdot,T\right)  \right\Vert _{L^{2}\left(
\omega\right)  } + \mathrm{e}^{D_{\ell}} \mathrm{e}^{-C_{2}\frac{1}{h}}\left\Vert U\left(  \cdot,0\right)
\right\Vert ^{1+M_{\ell}}.%
\end{equation}
For any $2\ell h\geq T$, $1\leq
\mathrm{e}^{D_{\ell}}\mathrm{e}^{C_{2}\frac{2\ell}{T}}\mathrm{e}^{-C_{2}\frac{1}{h}}$. Using the fact that $\left\Vert U\left(  \cdot,T\right)  \right\Vert \leq\left\Vert
U\left(  \cdot,0\right)  \right\Vert ,$ we deduce that for any
$2\ell h\geq T$,%
\begin{equation}\label{23.}
\left\Vert U\left(  \cdot,T\right)  \right\Vert ^{1+M_{\ell}}\leq \mathrm{e}^{D_{\ell}}
\mathrm{e}^{C_{1}\frac{1}{h}}\left\Vert U\left(  \cdot,0\right)  \right\Vert
^{M_{\ell}}\left\Vert u\left(  \cdot,T\right)  \right\Vert _{L^{2}\left(
\omega\right)  }+\mathrm{e}^{D_{\ell}}\mathrm{e}^{C_{2}\frac{2\ell}{T}}\mathrm{e}^{-C_{2}\frac{1}{h}}\left\Vert
U\left(\cdot,0\right)  \right\Vert ^{1+M_{\ell}}\text{ .}%
\end{equation}
 Using \eqref{42} and \eqref{23.}, for any $h >0$, we obtain that
\begin{equation*}
\left\Vert U\left(  \cdot,T\right)  \right\Vert ^{1+M_{\ell}}\leq \mathrm{e}^{D_{\ell}}
\mathrm{e}^{C_{1}\frac{1}{h}}\left\Vert U\left(  \cdot,0\right)  \right\Vert
^{M_{\ell}}\left\Vert u\left(  \cdot,T\right)  \right\Vert _{L^{2}\left(
\omega\right)  }+\mathrm{e}^{D_{\ell}}\mathrm{e}^{C_{2}\frac{2\ell}{T}}\mathrm{e}^{-C_{2}\frac{1}{h}}\left\Vert
U\left(\cdot,0\right)  \right\Vert ^{1+M_{\ell}}\text{ .}%
\end{equation*}
Finally, we choose $h>0,$ such that
\[
\mathrm{e}^{D_{\ell}}\mathrm{e}^{C_{2}\frac{2\ell}{T}}\mathrm{e}^{-C_{2}\frac{1}{h}}\left\Vert U\left(
\cdot,0\right)  \right\Vert ^{1+M_{\ell}}=\frac{1}{2}\left\Vert U\left(
\cdot,T\right)  \right\Vert ^{1+M_{\ell}}\text{ ,}%
\]
that is,
\[
\mathrm{e}^{C_{2}\frac{1}{h}}=2\mathrm{e}^{D_{\ell}}\mathrm{e}^{C_{2}\frac{2\ell}{T}}\left(  \frac{\left\Vert
U\left(  \cdot,0\right)  \right\Vert }{\left\Vert U\left(  \cdot,T\right)
\right\Vert }\right)^{1+M_{\ell}},%
\]
in order that%
\[
\left\Vert U\left(  \cdot,T\right)  \right\Vert^{1+M_{\ell}}\leq 2\mathrm{e}^{D_{\ell}}\left(
2\mathrm{e}^{D_{\ell}}\mathrm{e}^{C_{2}\frac{2\ell}{T}}\left(  \frac{\left\Vert U\left(\cdot,0\right)
\right\Vert }{\left\Vert U\left(\cdot,T\right)  \right\Vert }\right)
^{1+M_{\ell}}\right)  ^{\frac{C_{1}}{C_{2}}}\left\Vert U\left(  \cdot
,0\right)  \right\Vert ^{M_{\ell}}\left\Vert u\left(  \cdot,T\right)
\right\Vert _{L^{2}\left(  \omega\right)  }.
\]
Hence
\[
\left\Vert U\left(  \cdot,T\right)  \right\Vert^{1+M_{\ell}+\left(1+M_{\ell}\right) \frac{C_{1}}{C_{2}}} \leq2^{1+\frac{C_{1}}{C_{2}}}%
\mathrm{e}^{D_{\ell}\left(1+\frac{C_{1}}{C_{2}}\right)}\mathrm{e}^{C_{1}\frac{2\ell}{T}}\left(\left\Vert U\left(\cdot,0\right)
\right\Vert \right)
^{M_{\ell}+\left(  1+M_{\ell}\right)  \frac{C_{1}}{C_{2}}}\left\Vert u\left(
\cdot,T\right)  \right\Vert _{L^{2}\left(  \omega\right)  }.
\]
Setting $\sigma =M_{\ell}+\left(1+M_{\ell}\right) \frac{C_{1}}{C_{2}},$ $\mu =  2^{1+\frac{C_{1}}{C_{2}}}%
\mathrm{e}^{D_{\ell}\left(1+\frac{C_{1}}{C_{2}}\right)}$ and $k = 2\ell c_{1},$ we obtain
\[
\left\Vert U\left(  \cdot,T\right) \right\Vert \leq \left( \mu \mathrm{e}^{\frac{k}{T}}\right)^{\frac{1}{1+\sigma}}  \left\Vert U\left(\cdot,0\right) \right\Vert^{\frac{\sigma}{1+\sigma}}\left\Vert u\left(
\cdot,T\right)  \right\Vert _{L^{2}\left(  \omega\right)}^{\frac{1}{1+\sigma}  }\text{ .}%
\]
This  provides  the desired inequality.\endproof

Now, we state the following lemma, which is a direct corollary of Theorem \ref{thm1.1}, needed to establish the impulse controllability of system \eqref{1.1}.
\begin{lemma}\label{lemma3.1}
Let $\vartheta=\left(\upsilon,\upsilon_{\Gamma}\right)$ be the solution of \eqref{1.3}. Then there exist positive constants $\mathcal{M}_{1}$, $\mathcal{M}_{2}$ and $\delta=\delta\left(\Omega,\omega\right)$ such that the following estimate holds
\begin{equation}\label{2.38.}
\left\Vert \vartheta \left(\cdot,T\right)\right\Vert^2\leq \left(\frac{\mathcal{M}_{1}\mathrm{e}^{\frac{\mathcal{M}_{2}}{T}}}{\varepsilon^{\delta}}\right)^{2}\left\Vert \upsilon\left(\cdot,T\right)\right \Vert^2_{L^2(\omega)}+\varepsilon^2\left \Vert \vartheta^0 \right\Vert ^2.
\end{equation}
\end{lemma}
\begin{proof}
Using Theorem \ref{thm1.1}, there exist $\mathcal{K}_{1}>0$, $K_{2}>0$ and $\beta\in \left(0,1\right)$ such that
\begin{equation*}
\left\Vert \vartheta\left(\cdot,T\right)\right\Vert^2\leq \left(\mathcal{K}_{1}\mathrm{e}^\frac{\mathcal{K}_{2}}{T}\right)^2 \left\Vert \upsilon\left(\cdot,T\right)\right\Vert_{L^2(\omega)}^{2\beta}\left \Vert \vartheta\left(\cdot,0\right)\right\Vert^{2(1-\beta)}.
\end{equation*}
Let $\varepsilon>0$. We have 
\begin{align*}
\left(\mathcal{K}_{1}\mathrm{e}^\frac{\mathcal{K}_{2}}{T}\right)^2 \left\Vert \upsilon\left(\cdot,T\right)\right\Vert_{L^2(\omega)}^{2\beta}\left \Vert \vartheta\left(\cdot,0\right)\right\Vert^{2(1-\beta)}&=\left(\left(\mathcal{K}_{1}\mathrm{e}^\frac{\mathcal{K}_{2}}{T}\right)^\frac{1}{\beta}\left\Vert\upsilon(\cdot,T)\right\Vert_{L^2(\omega)}\frac{1}{\varepsilon^\frac{1-\beta}{\beta}}\left(1-\beta\right)^\frac{1-\beta}{2\beta}\right)^{2\beta}\\
&\quad \times \left( \varepsilon \left(\frac{1}{1-\beta}\right)^\frac{1}{2}\left\Vert \vartheta\left(\cdot,0\right)\right\Vert\right)^{2(1-\beta)}.
\end{align*}
Applying Young's inequality, we obtain
\[
\begin{array}[c]{ll}
\left(\mathcal{K}_{1}\mathrm{e}^\frac{\mathcal{K}_{2}}{T}\right)^2 \left\Vert \upsilon\left(\cdot,T\right)\right\Vert_{L^2(\omega)}^{2\beta}\left \Vert \vartheta\left(\cdot,0\right)\right\Vert^{2(1-\beta)}&\leq \left(\frac{\left(\mathcal{K}_{1}\mathrm{e}^\frac{\mathcal{K}_{2}}{T}\right)^\frac{1}{\beta}(1-\beta)^\frac{1-\beta}{2\beta}}{\varepsilon^\frac{1-\beta}{\beta}}\right)^2\beta \left\Vert \upsilon \left( \cdot,T\right)\right\Vert _{L^2(\omega)}^2\\
&\quad\quad\quad\quad+\varepsilon^2\left\Vert \vartheta(\cdot,0)\right\Vert^2\text{.}
\end{array}
\]
Thus,
\begin{equation*}
\left\Vert \vartheta \left(\cdot,T\right)\right\Vert^2\leq
\left(\frac{\left(\mathcal{K}_{1}\mathrm{e}^\frac{\mathcal{K}_{2}}{T}\right)^\frac{1}{\beta}(1-\beta)^\frac{1-\beta}{2\beta}}{\varepsilon^\frac{1-\beta}{\beta}}\right)^2\beta \left\Vert \upsilon \left( \cdot,T\right)\right\Vert _{L^2(\omega)}^2+\varepsilon^2\left\Vert \vartheta(\cdot,0)\right\Vert^2\text{.}
\end{equation*}
Therefore, we get our desired estimate \eqref{2.38.} with
$$
\mathcal{M}_{1}:=\mathcal{K}_{1}^{\frac{1}{\beta}}(1-\beta)^{\frac{1-\beta}{2 \beta}} \beta^{\frac{1}{2}} ; \quad \mathcal{M}_{2}:=\frac{\mathcal{K}_{2}}{\beta} ;\quad \delta:=\frac{(1-\beta)}{\beta}.
$$
\end{proof}

\section{Null approximate impulse controllability}
As an application of the inequality \eqref{1.2}, we study the null approximate impulse controllability of the heat equation with dynamic boundary conditions \eqref{1.1}, where the control function acts on a subdomain $\omega$ and at one point of time $\tau \in (0, T)$, (see more about impulse control in \cite{MBEYR,YT}).

\begin{definition}[see \cite{VTMN}] 
System \eqref{1.1} is null approximate impulse controllable at time $T$ if for any $\varepsilon > 0$ and
any $\Psi^0=\left(\psi^{0},\psi^{0}_{\Gamma}\right) \in \mathbb{L}^2$, there exists a control function $h \in L^2(\omega),$  such that the associated state at final time satisfies
\begin{equation*}
\|\Psi(\cdot, T)\| \leq \varepsilon\left\Vert \Psi^0\right\Vert .
\end{equation*}
\end{definition}
\noindent This means that for every $\varepsilon >0$ and $\Psi^0=\left(\psi^{0},\psi^{0}_{\Gamma}\right) \in \mathbb{L}^2$, the set
\begin{equation*}
\mathcal{R}_{T, \Psi^{0}, \varepsilon} :=\left\{h \in L^{2}(\omega): \text { the solution of }\eqref{1.1}\text { satisfies }\left\Vert\Psi(\cdot, T)\right\Vert \leq \varepsilon\left\Vert \Psi^{0}\right\Vert\right\},
\end{equation*}
is nonempty; which leads  to the definition of the cost of null approximate impulse control.
\begin{definition}   
The quantity  $$ K(T,\varepsilon):=\sup_{\left\|\Psi^0\right\|=1} \inf_{h \in \mathcal{R}_{T, \Psi^{0}, \varepsilon}} \|h\|_{L^{2}(\omega)}
$$ is called the cost of null
approximate impulse control at time $T.$
\end{definition}
Now, we are ready to prove the main result on null approximate impulse controllability for the system \eqref{1.1}.

\begin{proof}[Proof of Theorem \ref{thm1.2}]
Consider the following system
\begin{empheq}[left = \empheqlbrace]{alignat=2}\label{3.41}
\begin{aligned}
&\partial_{t} \upsilon-\Delta \upsilon=0, && \qquad \text { in } \Omega \times(0, T), \\
&\partial_{t}\upsilon_{\Gamma} - \Delta_{\Gamma} \upsilon_{\Gamma} + \partial_{\nu}\upsilon =0, && \qquad \text { on } \Gamma \times(0, T), \\
& \upsilon_\Gamma =\upsilon_{|\Gamma}, && \qquad \text { on } \Gamma \times(0, T), \\
&\left(\upsilon\left(\cdot,0\right),\upsilon_{\Gamma}\left(\cdot,0\right)\right)=\left(\upsilon^{0},\upsilon^{0}_{\Gamma}\right)=:\vartheta^{0}, && \qquad\text { on } \Omega\times\Gamma.\\
\end{aligned}
\end{empheq}
Let us fix $\varepsilon>0,$ $\Psi^{0}\in \mathbb{L}^2$,  and put $\kappa := \frac{\mathcal{M}_{1} \mathrm{e}^{\frac{\mathcal{M}_{2}}{T-\tau}}}{\varepsilon^{\delta}}$ where the constants $\mathcal{M}_1$, $\mathcal{M}_2$ and $\delta$ are from the estimation \eqref{2.38.}. We define the functional $J_{\varepsilon}: \mathbb{L}^2 \rightarrow \mathbb{R}$ as follows
\begin{equation*}
J_{\varepsilon}\left(\vartheta^{0}\right)=\frac{\kappa^{2}}{2}\|\upsilon(\cdot, T-\tau)\|_{L^{2}(\omega)}^{2}+\frac{\varepsilon^{2}}{2}\left\|\vartheta^{0}\right\| + \left\langle \Psi^{0}, \vartheta(\cdot,T) \right\rangle,
\end{equation*}
where $\vartheta=(\upsilon,\upsilon_\Gamma)$ is the solution of \eqref{3.41}.
Notice that $J_{\varepsilon}$ is strictly convex, $C^1$ and coercive, i.e.,  $J_{\varepsilon}\left(\vartheta^{0}\right) \rightarrow \infty$ when $\left\|\vartheta^{0}\right\| \rightarrow \infty$. Therefore, $J_{\varepsilon}$ has a unique minimizer $\tilde{\vartheta}^{0} \in \mathbb{L}^2$ such that 
$$J_{\varepsilon}\left(\tilde{\vartheta}^{0}\right)=\min\limits_{\vartheta^{0} \in \mathbb{L}^2} J_{\varepsilon}\left(\vartheta^{0}\right).$$ 
It implies that $J_{\varepsilon}^{\prime}\left(\tilde{\vartheta}^{0}\right) \zeta^{0}=0$ for any $\zeta^{0} \in \mathbb{L}^2,$ i.e., the following estimation holds for any $\zeta^{0}$
\begin{equation}\label{4.39}
\kappa^{2} \int_{\omega} \tilde{\upsilon}(x, T-\tau) z(x, T-\tau) \mathrm{d} x+\varepsilon^{2} \left\langle \tilde{\vartheta}^{0},\zeta^{0} \right\rangle+\left\langle \Psi^{0}, \zeta(\cdot,T) \right\rangle = 0,
\end{equation}
where $\tilde{\vartheta}=\left(\tilde{\upsilon},\tilde{\upsilon_\Gamma}\right)$ and $\zeta=(z,z_{\Gamma})$ are respectively the solutions of \eqref{3.41} corresponding to $\tilde{\vartheta}^{0}$ and $\zeta^{0}$.
Recall that $(\psi,\psi_{\Gamma})$ satisfies
\begin{empheq}[left = \empheqlbrace]{alignat=2} \label{5.1}
\begin{aligned}
&\partial_{t} \psi-\Delta \psi=0, && \qquad\text { in } \Omega \times(0, T) \backslash\{\tau\},\\
&\psi(\cdot, \tau)=\psi\left(\cdot, \tau^{-}\right)+\mathds{1}_{\omega} h(\cdot,\tau), && \qquad\text { in } \Omega,\\
&\partial_{t}\psi_{\Gamma} - \Delta_{\Gamma} \psi_{\Gamma} + \partial_{\nu}\psi =0, && \qquad\text { on } \Gamma \times(0, T)\backslash\{\tau\}, \\
& \psi_{\Gamma}(\cdot,t) = \psi_{|\Gamma}(\cdot,t), &&\qquad\text{ on } \Gamma \times(0, T) , \\
&\psi_{\Gamma}(\cdot, \tau)=\psi_{\Gamma}\left(\cdot, \tau^{-}\right), && \qquad\text { on } \Gamma.
\end{aligned}
\end{empheq}
Multiplying \eqref{5.1}$_{1}$ by $z(\cdot,T-t)$ and \eqref{5.1}$_{3}$ by $z_{\Gamma}(\cdot,T-t)$, for every $t\in (0,T) \backslash\{\tau\}$, we obtain
\[
\begin{array}{ll}
&\displaystyle\int_{\Omega}\partial_{t}\psi(x,t)z(x,T-t) \mathrm{d}x - \int_{\Omega}\Delta\psi(x,t)z(x,T-t) \mathrm{d}x 
+\int_{\Gamma}\partial_{t}\psi_{\Gamma}(x,t)z_{\Gamma}(x,T-t) \mathrm{d}S \\
&-\displaystyle \int_{\Gamma}\Delta_{\Gamma}\psi_{\Gamma}(x,t)z_{\Gamma}(x,T-t) \mathrm{d}S+\int_{\Gamma}\partial_{\nu}\psi(x,t)z_\Gamma(x,T-t) \mathrm{d}S =0.
\end{array}
\]
By two integration per parts, we obtain
\[
\begin{array}[c]{ll}
&\displaystyle\int_{\Omega}\partial_{t}\psi(x,t)z(x,T-t) \mathrm{d}x - \int_{\Gamma}\partial_{\nu}\psi(x,t)z_{\Gamma}(x,T-t)\mathrm{d}S+\int_{\Gamma}\psi_{\Gamma}(x,t)\partial_{\nu}z(x,T-t)\mathrm{d}S  \\
&-\displaystyle\int_{\Omega}\psi(x,t)\Delta z(x,T-t) \mathrm{d}x +
\int_{\Gamma}\partial_{t}\psi_{\Gamma}(x,t)z_{\Gamma}(x,T-t) \mathrm{d}S\\
&- \displaystyle\int_{\Gamma}\psi_{\Gamma}(x,t)\Delta_{\Gamma} z_{\Gamma}(x,T-t) \mathrm{d}S+\int_{\Gamma}\partial_{\nu}\psi(x,t)z_\Gamma(x,T-t) \mathrm{d}S =0.
\end{array}
\]
Hence
\[
\begin{array}[c]{ll}
&\displaystyle\int_{\Omega}\partial_{t}\psi(x,t)z(x,T-t) \mathrm{d}x-\int_{\Omega}\psi(x,t)\Delta z(x,T-t) \mathrm{d}x  +
\int_{\Gamma}\partial_{t}\psi_{\Gamma}(x,t)z_{\Gamma}(x,T-t) \mathrm{d}S\\
&-\displaystyle\int_{\Gamma}\psi_{\Gamma}(x,t)\left(\Delta_{\Gamma}z_{\Gamma}(x,t)-\partial_{\nu}z(x,T-t)\right)\mathrm{d}S=0\text{.}  
\end{array}
\]
Since $\zeta=(z,z_{\Gamma})$ is the solution of \eqref{3.41}, then $\Delta z(\cdot,T-t)=\partial_{t}z(\cdot,T-t)$ and $$\Delta_{\Gamma}z_{\Gamma}(\cdot,T-t)-\partial_{\nu}z(\cdot,T-t)=\partial_{t}z_{\Gamma}(\cdot,T-t).$$
Therefore,
\[
\begin{array}[c]{ll}
&\displaystyle\int_{\Omega}\partial_{t}\psi(x,t)z(x,T-t) \mathrm{d}x-\int_{\Omega}\psi(x,t)\partial_{t} z(x,T-t) \mathrm{d}x  +
\int_{\Gamma}\partial_{t}\psi_{\Gamma}(x,t)z_{\Gamma}(x,T-t) \mathrm{d}S\\
&-\displaystyle\int_{\Gamma}\psi_{\Gamma}(x,t)\partial_{t}z(x,T-t)\mathrm{d}S=0.
\end{array}
\]
That is, 
\begin{equation}\label{3.44}
  \displaystyle\int_{\Omega}\frac{d}{d\mathrm{t}}\bigg( \psi(x,t)z(x,T-t) \bigg) \mathrm{d}x + \int_{\Gamma}\frac{d}{d \mathrm{t}} \bigg( \psi_{\Gamma}(x,t)z_{\Gamma}(x,T-t)\bigg) \mathrm{d}S = 0.
\end{equation}
Integrating \eqref{3.44} over $(0, \tau)$ yields
\begin{equation*}
\int_{\Omega}\left[\psi(x,t)z(x,T-t)\right]_{0}^{\tau} \mathrm{d} x+\int_{\Gamma}\left[\psi_{\Gamma}(x,t)z_{\Gamma}(x,T-t)\right]_{0}^{\tau} \mathrm{d}S=0.
\end{equation*}
Therefore, 
\begin{equation}\label{4.41}
\int_{\Omega} \left(\psi(x,\tau^{-})z(x,T-\tau) -\psi(x,0)z(x,T)\right) \mathrm{d}x+\int_{\Gamma}\left(\psi_{\Gamma}(x,\tau^{-})z_{\Gamma}(x,T-\tau) -\psi_{\Gamma}(x,0)z_{\Gamma}(x,T)\right) \mathrm{d}S=0.
\end{equation}
Integrating \eqref{3.44} over $(\tau, T)$ gives us
\begin{equation*}
\int_{\Omega}\left[\psi(x,t)z(x,T-t)\right]_{\tau}^{T} \mathrm{d}x+\int_{\Gamma}\left[\psi_{\Gamma}(x,t)z_{\Gamma}(x,T-t)\right]_{\tau}^{T} \mathrm{d}S=0.
\end{equation*}
Hence
\begin{equation}\label{4.42}
\int_{\Omega}\left(\psi(x,T)z(x,0) -\psi(x,\tau)z(x,T-\tau)\right) \mathrm{d}x+\int_{\Gamma}\left(\psi_{\Gamma}(x,T)z_{\Gamma}(x,0) -\psi_{\Gamma}(x,\tau)z_{\Gamma}(x,T-\tau)\right) \mathrm{d}S=0.
\end{equation}
Combining \eqref{4.41}-\eqref{4.42}, and taking in consideration the fact that $\psi(\cdot, \tau)=\psi\left(\cdot, \tau^{-}\right)+\mathds{1}_{\omega} h(\tau)$ and $\psi_{\Gamma}(\cdot, \tau)=\psi_{\Gamma}\left(\cdot, \tau^{-}\right)$, we obtain
\begin{equation}\label{4.43}
\begin{aligned}
&\int_{\omega} h(\tau,x)z(x,T-\tau) \mathrm{d} x+\int_{\Omega}\psi(x,0)z(x,T)\mathrm{d} x +\int_{\Gamma}\psi_{\Gamma}(x,0)z_{\Gamma}(x,T) \mathrm{d}S \\
&-\int_{\Omega} \psi(x,T)z(x,0) \mathrm{d}x -\int_{\Gamma} \psi_{\Gamma}(x,T)z_{\Gamma}(x,0) \mathrm{d}S = 0.
\end{aligned}
\end{equation}
Thus, if we choose $h(\tau,x)=\kappa^{2} \tilde{v}(x, T-\tau)$, we obtain, from \eqref{4.39} and \eqref{4.43}, that
\begin{equation*}
\left\langle \Psi(\cdot, T) +\varepsilon^{2} \tilde{\vartheta}^{0}, \zeta^{0} \right\rangle =0, \qquad \forall \zeta^{0} \in \mathbb{L}^2.
\end{equation*}
Hence, $\Psi(x, T) = -\varepsilon^{2} \tilde{\vartheta}^{0}(x).$ Moreover, with $\zeta^{0} \equiv \tilde{\vartheta}^{0},$ using the Cauchy-Schwarz inequality, it follows from \eqref{4.39} that
\begin{equation*}
\kappa^{2}\|\tilde{\upsilon}(\cdot, T-\tau)\|_{L^{2}(\omega)}^{2}+\varepsilon^{2}\left\Vert\tilde{\vartheta}^{0}\right\Vert ^2 \leq\left\|\Psi^{0}\right\| \|\tilde{\vartheta}(\cdot, T)\|.
\end{equation*}
By virtue of the energy estimate for the system \eqref{3.41}, which is
\begin{equation*}
\|\tilde{\vartheta}(\cdot, T)\| \leq\|\tilde{\vartheta}(\cdot, T-\tau)\|,
\end{equation*}
we obtain
\begin{equation}\label{3.51}
\kappa^{2}\|\tilde{\upsilon}(\cdot, T-\tau)\|_{L^{2}(\omega)}^{2}+\varepsilon^{2}\left\|\tilde{\vartheta}^{0}\right\| ^2\leq\left\|\Psi^{0}\right\| \|\tilde{\vartheta}(\cdot, T-\tau)\|.
\end{equation} 
Applying the result in Lemma \ref{lemma3.1}, which is
\begin{equation*}
\|\tilde{\vartheta}(\cdot, T-\tau)\|^{2} \leq \kappa^{2}\|\tilde{\upsilon}(\cdot, T-\tau)\|_{L^{2}(\omega)}^{2}+\varepsilon^{2}\|\tilde{\vartheta}(\cdot, 0)\|^{2},
\end{equation*}
and using \eqref{3.51}, we obtain
\begin{equation}\label{3.53}
\|\tilde{\upsilon}(\cdot, T-\tau)\|_{L^{2}(\omega)} \leq\left\|\Psi^{0}\right\|.
\end{equation}
Finally, combining \eqref{3.51} and \eqref{3.53}, we obtain
\begin{equation*}
\kappa^{2}\|\tilde{\upsilon}(\cdot, T-\tau)\|_{L^{2}(\omega)}^{2}+\varepsilon^{2}\left\|\tilde{\vartheta}^{0}\right\| ^2\leq\left\|\Psi^{0}\right\|^2.
\end{equation*}
Recall that $\Psi(x, T) = -\varepsilon^{2} \tilde{\vartheta}^{0}(x)$ and $h(x)=\kappa^{2} \tilde{v}(x, T-\tau)$. Thus
\begin{equation*}
\frac{1}{\kappa^{2}}\|h\|_{L^{2}(\omega)}^{2}+\frac{1}{\varepsilon^{2}}\|\Psi(\cdot, T)\|^{2} \leq\left\|\Psi^{0}\right\|^{2}.
\end{equation*}
This completes the proof.
\end{proof}

\section{Conclusions and Remarks}
In this work, the impulsive null approximate controllability for the heat equation with dynamic boundary conditions is proved for a bounded convex domain $\Omega \subset \mathbb{R}^n$ with smooth boundary $\Gamma$. The used technique is based on a logarithmic convexity estimate obtained by a Carleman commutator approach. The results of this work can be generalized to the case of a star-shaped domain or the general case of a $C^2$ domain by adapting the approach given in \cite{pkm} for Dirichlet boundary conditions.

Finally, in real life problems, the impulse starts abruptly at a certain moment of time and remains active on a finite time interval. However, the time of the action is little. Such an impulse is known as non-instantaneous impulse \textbf{(NII)} (see \cite{LHWZME}). It would be of much interest to investigate the impulsive null approximate controllability for the system \eqref{1.1} with  non-instantaneous impulses.

\end{document}